\documentclass{article}

\usepackage[latin1]{inputenc}

\usepackage{pgf,tikz}
\usetikzlibrary{arrows}
\usetikzlibrary{patterns}

\usepackage{subfigure}

\usepackage{amsmath,amssymb,amsthm,epsfig}

\def\blue{\textcolor{blue}}
\newcommand*{\doi}[1]{\href{http://dx.doi.org/\detokenize{#1}}{\blue{doi: \detokenize{#1}}}}

\def\la{\leftarrow}
\def\ra{\rightarrow}
\def\ua{\uparrow}
\def\da{\downarrow}

\newtheorem{Theorem}{Theorem}[section]
\newtheorem{Corollary}[Theorem]{Corollary}

\newtheorem{Lemma}[Theorem]{Lemma}
\newtheorem{Definition}[Theorem]{Definition}
\newtheorem{appxlem}{Lemma}[section]

\numberwithin{equation}{section}

\newcommand{\N}{{\mathbb N}}

\newcommand{\beq}{\begin{equation}}
\newcommand{\eeq}{\end{equation}}

\usepackage{multirow}

\def\mystrut{\hbox{\vrule height12.0pt depth2pt width0pt}}
\def\mybox{\hbox to 12.0pt}
\def\norulefill{\leaders\hrule height0pt\hfill}
\def\nr#1{\multispan{#1}\norulefill}
\def\hr#1{\multispan{#1}\hrulefill}

\def\mybigbox{\hbox to 35.0pt}
\def\myverybigbox{\hbox to 60.0pt}

\def\sc{\scriptstyle}
\def\ds{\displaystyle}    
\def\ov{\overline}
\def\p{{\bf{d}}}
\def\s{{\bf{s}}}

\def\WE{\textrm{WE}}
\def\NS{\textrm{NS}}
\def\NE{\textrm{NE}}
\def\SE{\textrm{SE}}
\def\NW{\textrm{NW}}
\def\SW{\textrm{SW}}
\def\XY{\textrm{XY}}

\def\wgt{{\rm wgt}}
\def\neg{{\rm neg}}

\usepackage{hyperref}

\def\ch{{\rm ch}\,}
\def\x{\boldsymbol{x}}
\def\y{\boldsymbol{y}}
\def\z{\boldsymbol{z}}
\def\q{\boldsymbol{q}}
\def\t{\boldsymbol{t}}
\def\u{\boldsymbol{u}}
\def\a{\boldsymbol{a}}
\def\b{\boldsymbol{b}}
\def\c{\boldsymbol{c}}

\def\1{\boldsymbol{1}}

\begin{document}


\title{Half-turn symmetric alternating sign matrices and Tokuyama type factorisation
for orthogonal group characters}

\author{Ang\`ele M. Hamel\thanks{e-mail: ahamel@wlu.ca} \\
Department of Physics and Computer Science,\\
Wilfrid Laurier University, 
Waterloo, Ontario N2L 3C5, Canada\\
\\
and \\
\\
Ronald C. King\thanks{e-mail: R.C.King@maths.soton.ac.uk} \\
School of Mathematics, University of Southampton, \\
Southampton SO17 1BJ, England \\
}

\maketitle

\begin{abstract}
Half-turn symmetric alternating sign matrices (HTSASMs) are special variations of the well-known alternating sign matrices which have a long and fascinating history.  HTSASMs are interesting combinatorial objects in their own right and have been the focus of recent study.  Here we explore counting weighted HTSASMs with respect to a number of statistics to derive an orthogonal group version of Tokuyama's factorisation formula, which involves a deformation and expansion of Weyl's denominator formula multiplied by a general linear group character.  Deformations of  Weyl's original denominator formula to other root systems have been discovered by Okada and Simpson, and it is thus natural to ask for versions of Tokuyama's  factorisation formula involving 
other root systems. Here we obtain such a formula involving a deformation of Weyl's denominator formula for the  orthogonal group multiplied by a deformation of an orthogonal group character.
\end{abstract}


\section{Introduction}
\label{sec:in}

Half-turn symmetric alternating sign matrices (HTSASMs) are $N \times N$ matrices that are invariant under a 180 degree rotation and that contain as entries only $0, 1,$ or $-1$, arranged such that each row and column sums to $1$, while all the partial row and column sums are either $1$ or $0$. HTSASMs are a variation on the well-known alternating sign matrices which have a long and fascinating history  \cite{Bressoud}. HTSASMs are interesting combinatorial objects in their own right, and while they occur in some of the earliest papers on ASMs (Robbins \cite{Robbins}, Kuperberg \cite{Kuperberg}, Okada \cite{Okada}), they are also the focus of recent study: see, for example, \cite{Aval, Kuperberg, Okada, Razumov, Razumov1, Tabony}, among others.  Here we take a view closest to that of Okada \cite{Okada} and Simpson \cite{Simpson} and explore counting HTSASMs weighted with respect to a number of statistics to derive generalisations of deformations of Weyl's denominator formula in what Okada called the $B_n$ case and in what Simpson called the $B_n'$ case.

In order to do this we go further and derive corresponding generalisations of Tokuyama's identity~\cite{Tokuyama}. 
Taking advantage of the bijective correspondence between strict Gelfand-Tsetlin patterns with top row specified by a strict partition $\lambda$ 
of length $n$ and corresponding generalisations of ASMs known as $\lambda$-ASMs~\cite{Okada}, Tokuyama's identity can be expressed 
in the form:
\begin{equation}\label{eqn-tok-lambda}
          \sum_{A\in{\cal A}^\lambda_n}\ \wgt(A)  =  \sum_{A\in{\cal A}^\delta_n}\ \wgt(A)  \  \  s_\mu(\x)
\end{equation}
with
\begin{equation}\label{eqn-tok-delta}
           \sum_{A\in{\cal A}^\delta_n}\ \wgt(A)  = \prod_{1\leq i<j\leq n} (1 + t\, x_ix_j^{-1} )
\end{equation}
where $t$ is an arbitrary parameter embodied in each summation through the $t$-dependent specification of $\wgt(A)$ provided by Tokuyama.
In these equations $\lambda=\mu+\delta$,with $\delta$ the strict partition $(n,n-1,\ldots,1)$ and $\mu$ an ordinary partition with no
more than $n$ non-zero parts. The Schur function $s_\mu(x)$ is none other than the character of the irreducible representation of 
the general linear group $GL(n)$ of highest weight $\mu$ evaluated for a group element with eigenvalues $\x=(x_1,x_2,\ldots,x_n)$. 

The identity (\ref{eqn-tok-delta}) represents a deformation of Weyl's denominator 
formula for the reductive Lie algebra $gl(n)$ of the general linear group $GL(n)$ in the sense that on setting $t=-1$ 
and $x_i=e^{-\epsilon_i}$ for $i=1,2,\ldots,n$ one recovers Weyl's two formulae. In particular with this specialisation
the right hand side of (\ref{eqn-tok-delta}) takes the form $\prod_{\alpha\in\Delta_+} (1-e^{-\alpha})$ in which $\Delta_+$ is the set of 
all positive roots of $gl(n)$, that is $\alpha=\epsilon_i-\epsilon_j$ for $1\leq i<j\leq n$.
Similarly, the identity (\ref{eqn-tok-lambda}) represents a deformation of Weyl's character formula 
for $gl(n)$ since it can be shown that the two summations each reduce to sums over elements in the Weyl group of $gl(n)$;
that is, the symmetric group $S_n$. 

Deformations of  Weyl's denominator formula for other root systems analogous to (\ref{eqn-tok-delta})
were discovered by Okada \cite{Okada} and Simpson \cite{Simpson}, 
see for example Theorems ~\ref{the-Okada} and ~\ref{the-Simpson} below.
It is therefore natural to ask for analogues of Tokuyama's identity (\ref{eqn-tok-lambda})
for other root systems.
Some cases are already known, and previous work includes  \cite{Brubaker1}, \cite{Brubaker2}, \cite{Brubaker3}, \cite{Chinta}, \cite{HKWeyl}.

The impetus for our approach to this problem comes from the 2010  BIRS workshop ``Whittaker Functions, Crystal Bases, and Quantum Groups'' at which Ben Brubaker asked the first author whether we could discover a Tokuyama type factorisation formula based
on orthogonal group shifted tableaux
along the same lines as our previous work \cite{HKWeyl} on symplectic shifted tableaux.
He also made us aware of Tabony's~\cite{Tabony} progress in this direction using HTSASMs,
see Theorem~\ref{the-Tabony} below. 
From this query we developed a number of conjectures~\cite{HKFPSAC} for certain
$B_n,\; B'_n, \; C_n,\; C'_n,\; D_n,$ and $D'_n$ cases, some of which will be the subject of forthcoming papers,
while here we present proofs of our 
Tokuyama type identities 
in the cases $B_n$ and $B'_n$ associated with the root system of the Lie algebra of the orthogonal group $SO(2n+1)$.

The paper is organized as follows:  
Section~\ref{sec:htsasm} covers  definitions of the various types of ASMs including $\lambda$-HTSASMs. A number of known theorems regarding weighted sums of
these ASMs are gathered together in Section~\ref{sec:ASMweights} together with the statement of a new multi-parameter theorem, Theorem~\ref{the-result1},
whose proof will emerge only later. To establish this proof it is convenient to exploit two new combinatorial constructs; namely primed shifted tableaux and certain corresponding sets of non-intersecting lattice paths. These are introduced in Section~\ref{sec:PSTandLPs} and then appropriately weighted in
Section~\ref{sec:PSTandLPwgts}. Section~\ref{sec:LPresult} contains a lattice path proof of an intermediate theorem, Theorem~\ref{the-Pdet}, involving a determinant that is explicitly evaluated in Section~\ref{sec:det}. Following some preliminaries on Schur functions and orthogonal group characters in Section~\ref{sec:characters},
our key results are established in Section~\ref{sec:Tok} in the form of Theorems~\ref{the-result2} and~\ref{the-result3}. In  Section \ref{sec:corollaries} it is shown that corollaries of the main results include extensions of Tokuyama's factorisation identities (\ref{eqn-tok-lambda}) and (\ref{eqn-tok-delta}) 
to the ASMs of the $B'_n$ and $B_n$ cases, each case involving a deformation of an orthogonal group character. Further corollaries also include the theorems of Okada, Simpson, Tabony and our own Theorem~\ref{the-result1}. In a final Addendum it is shown that the universal weighting scheme of Brubaker and Schultz~\cite{BrubakerSchultz} can not only be accommodated in our formalism, but can be analysed in such a way that a hitherto unidentified quotient also emerges as a deformed character of the orthogonal group.   An Appendix provides an alternative derivation of one of the lemmas quoted in Section~\ref{sec:det}.

\section{Half-turn symmetric ASMs}
\label{sec:htsasm}


The setting for this work is a particular type of alternating sign matrix (ASM), variously described either as an ASM invariant under 180 degree rotation (Okada~\cite{Okada}) or as a half-turn symmetric ASM (Robbins~\cite{Robbins}, Kuperberg~\cite{Kuperberg}, Tabony~\cite{Tabony}).   
Using this perspective we can consider several types of HTSASMs, depending on whether the number of rows is even or odd, and depending on restrictions on the central row or column.  However, in this paper we consider only two types: 
even sided $2n\times 2n$ HTSASMs, and odd-sided $(2n+1)\times(2n+1)$ HTSASMs with a special central column consisting of a $1$ in the central $(n+1,n+1)$ position and zeroes elsewhere.
Since there is rotational symmetry we can reconstruct the entire ASM from its rightmost $n$ columns
Thus we will often use the right hand portion 
of the matrix, with a 
turn on the left hand boundary, so that for $i=1,2,\ldots,n$ row $i$ 
read from right to left continues as row $N+1-i$ read from left to right with $N=2n$ or $2n+1$ as appropriate.   

The following are examples of the two types:
{\bf Type ${\cal B}_n$} Even-sided half-turn symmetric ASM (Okada \cite{Okada}, Tabony  \cite{Tabony});
and {\bf Type ${\cal B}'_n$} Odd-sided half-turn symmetric ASM (special central column)~[Simpson~\cite{Simpson}].
Here and elsewhere, for alignment purposes, we represent an ASM entry $-1$ by $\ov1$.
\begin{equation}
\begin{array}{ccc}
\mbox{Type $A\in{\cal B}_3$:}&\qquad&\mbox{Type $A\in{\cal B}'_2$:}\cr\cr
\left[
\begin{array}{cccccc}
0 & 0 & 0  & 1 & 0 & 0 \\
0 & 1 & 0 &\ov1 & 0 & 1 \\
0 & 0 & 1 &  0 & 0 & 0\\
0 & 0 & 0 & 1 & 0 & 0  \\
1 & 0 &\ov1 & 0 & 1 & 0 \\
0 & 0 & 1 & 0 & 0 & 0 \\
\end{array}
\right]
&&
\left[
\begin{array}{ccccc}
0 & 1 & \bf0 & 0 & 0 \\
0 & 0 & \bf0 & 1 & 0 \\
1 & \ov{1} & \bf1 & \ov{1} & 1\\
0 & 1 &\bf0 & 0 & 0 \\
0 & 0 &\bf0 & 1 & 0\\
\end{array}
\right]
\end{array}
\label{asm-BnB'n}
\end{equation}

Just like ordinary ASMs \cite{Bressoud,HKUturn}, HTSASMs can also be realised 
as compass points matrices (CPMs). While entries $1$ and $-1$ in an ASM are mapped to $\WE$ and $\NS$, respectively, 
in the corresponding CPM, an entry $0$ in an ASM is mapped to one or other of the CPM entries $\NE$, $\SE$,
$\NW$ or $\SW$ in accordance with the compass point arrangements of their nearest non-vanishing neighbours as illustrated below:
\beq
\begin{array}{ccccccccccc} 
    \ov1   &    &    1     &      &    \ov1   &     &   1     &    &   \ov1   &    & 1\cr
\ov1~~{\bf1}~~\ov1&\ & 1~~{\bf\ov1}~~1 &\ &  \ov1~~{\bf0}~~1 &\ &1~~{\bf0}~~\ov1 &\ & 1~~{\bf0}~~\ov1&\ &\ov1~~{\bf0}~~1\cr
    \ov1   &    &    1     &      &   1       &     &   \ov1  &    &    1     &    &    \ov1 \cr
\cr
     \WE &    &\NS  &    &\SW      &    &\NE    &    &\SE     &    &\NW     \cr
\end{array}
\label{fig0}
\eeq

In the case of HTSASMs one can confine attention to the right hand portion. 
Doing this in the case of the above two examples, one obtains:
\begin{equation}
\begin{array}{cc}
\left[
\begin{array}{ccc}
 1 & 0 & 0 \\
\ov1 & 0 & 1 \\
 0 & 0 & 0\\
 1 & 0 & 0  \\
 0 & 1 & 0 \\
 0 & 0 & 0 \\
\end{array}
\right]
\mapsto
\left[
\begin{array}{cccccc}
 \sc{\WE} & \sc{\SE} & \sc{\SE} \\
 \sc{\NS} & \sc{\SW} & \sc{\WE} \\
 \sc{\SE} & \sc{\SE} & \sc{\NE}\\
 \sc{\WE} & \sc{\SE} & \sc{\NE}  \\
 \sc{\NW} & \sc{\WE} & \sc{\NE} \\
 \sc{\NE} & \sc{\NE} & \sc{\NE} \\
\end{array}
\right]
&
\left[
\begin{array}{ccccc}
 0 & 0 \\
 1 & 0 \\
 \ov{1} & 1\\
 0 & 0 \\
  1 & 0\\
\end{array}
\right]
\mapsto
\left[
\begin{array}{ccccc}
 \sc{\SE} & \sc{\SE} \\
 \sc{\WE} & \sc{\SE} \\
 \sc{\NS} & \sc{\WE}\\
 \sc{\SE} & \sc{\NE} \\
 \sc{\WE} & \sc{\NE}\\
\end{array}
\right]
\end{array}
\label{cpm-BnB'n}
\end{equation}

The compass point matrices are equivalent to the square ice configurations (SICs) 
that arise in the six vertex model in physics. Here there is a two--dimensional grid 
consisting of vertices and directed edges.  Each vertex has four edges, and the directions 
of the edges are such that only six different in-out arrangements 
are permitted.  Each of these types of vertex arrangement 
is in one-to-one correspondence with a type of entry in the ASM, and each 
can be uniquely labeled by the compass point directions of its two incoming edges. 
These labels are precisely those specifying the corresponding entry in the CPM.

\begin{equation}\label{fig1}
\vcenter{\hbox{
\begin{tikzpicture}[x={(0in,-0.2in)},y={(0.2in,0in)}] 
\foreach \j in {0,...,5}
\draw(1,2+3*\j)node{$\bullet$};
\draw(3,2)node{$1$};
\draw(3,5)node{$-1$};
\draw(3,8)node{$0$};
\draw(3,11)node{$0$};
\draw(3,14)node{$0$};
\draw(3,17)node{$0$}; 
\draw(4,2)node{$\WE$};
\draw(4,5)node{$\NS$};
\draw(4,8)node{$\SW$};
\draw(4,11)node{$\NE$};
\draw(4,14)node{$\SE$};
\draw(4,17)node{$\NW$}; 

\draw[very thick,->] (1,1)to(1,1.75); \draw[very thick] (1,1.75)to(1,2);
\draw[very thick] (1,2)to(1,2.25); \draw[very thick,<-] (1,2.25)to(1,3);
\draw[very thick](0,2)to(0.25,2); \draw[very thick,<-] (0.25,2)to(1,2);
\draw[very thick,->](1,2)to(1.75,2); \draw[very thick] (1.75,2)to(2,2);

\draw[very thick] (1,1+3)to(1,1.25+3); \draw[very thick,<-] (1,1.25+3)to(1,2+3);
\draw[very thick,->] (1,2+3)to(1,2.75+3); \draw[very thick] (1,2.75+3)to(1,3+3);
\draw[very thick,->](0,2+3)to(0.75,2+3); \draw[very thick] (0.75,2+3)to(1,2+3);
\draw[very thick](1,2+3)to(1.25,2+3); \draw[very thick,<-] (1.25,2+3)to(2,2+3);

\draw[very thick,->] (1,1+6)to(1,1.75+6); \draw[very thick] (1,1.75+6)to(1,2+6);
\draw[very thick,->] (1,2+6)to(1,2.75+6); \draw[very thick] (1,2.75+6)to(1,3+6);
\draw[very thick](0,2+6)to(0.25,2+6); \draw[very thick,<-] (0.25,2+6)to(1,2+6);
\draw[very thick](1,2+6)to(1.25,2+6); \draw[very thick,<-] (1.25,2+6)to(2,2+6);

\draw[very thick] (1,1+9)to(1,1.25+9); \draw[very thick,<-] (1,1.25+9)to(1,2+9);
\draw[very thick] (1,2+9)to(1,2.25+9); \draw[very thick,<-] (1,2.25+9)to(1,3+9);
\draw[very thick,->](0,2+9)to(0.75,2+9); \draw[very thick] (0.75,2+9)to(1,2+9);
\draw[very thick,->](1,2+9)to(1.75,2+9); \draw[very thick] (1.75,2+9)to(2,2+9);

\draw[very thick] (1,1+12)to(1,1.25+12); \draw[very thick,<-] (1,1.25+12)to(1,2+12);
\draw[very thick] (1,2+12)to(1,2.25+12); \draw[very thick,<-] (1,2.25+12)to(1,3+12);
\draw[very thick](0,2+12)to(0.25,2+12); \draw[very thick,<-] (0.25,2+12)to(1,2+12);
\draw[very thick](1,2+12)to(1.25,2+12); \draw[very thick,<-] (1.25,2+12)to(2,2+12);

\draw[very thick,->] (1,1+15)to(1,1.75+15); \draw[very thick] (1,1.75+15)to(1,2+15);
\draw[very thick,->] (1,2+15)to(1,2.75+15); \draw[very thick] (1,2.75+15)to(1,3+15);
\draw[very thick,->](0,2+15)to(0.75,2+15); \draw[very thick] (0.75,2+15)to(1,2+15);
\draw[very thick,->](1,2+15)to(1.75,2+15); \draw[very thick] (1.75,2+15)to(2,2+15);
\end{tikzpicture}
}}
\end{equation}

For a given HTSASM and its associated CPM, combining the designated vertex arrangements in a rectangular array 
yields a U-turn square ice configuration 
that is characterised by curving 
the edges on the left hand boundary in such a way that they join row $i$ to row $N-i+1$ 
for $i=1,2,\ldots,n$ in the even and odd rowed cases $N=2n$ and $N=2n+1$, as illustrated below 
in Figure~\ref{fig-ACS-delta}.

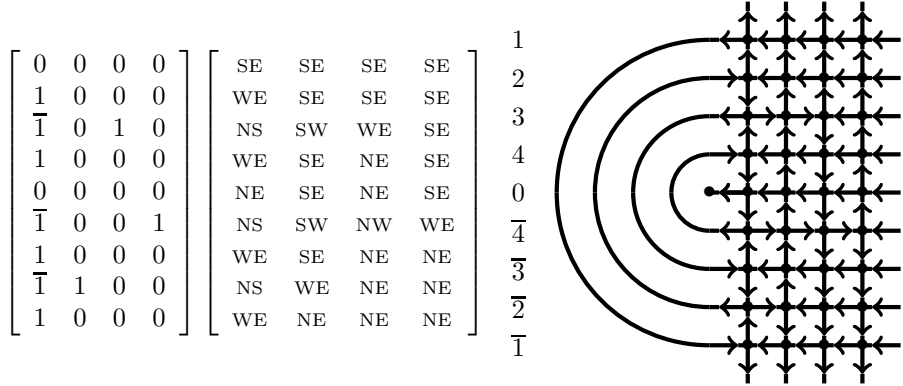
\begin{figure*}[htbp]
\begin{center}
\begin{equation*}\hskip-1ex
\begin{array}{c}
\left[\begin{array}{cccc}
0 & 0 & 0 & 0\\
1 & 0 & 0 & 0 \\
\ov1 & 0 & 1 & 0\\
1 & 0 & 0 & 0\\
0 & 0 & 0 & 0\\
\ov1 & 0 & 0 & 1\\
1 & 0 & 0 & 0\\
\ov1 & 1 & 0 &0 \\
1 & 0 & 0 & 0\\
\end{array}\right]
\left[\begin{array}{cccc}
\sc{\SE} &\sc{\SE} &\sc{\SE} &\sc{\SE}\\
\sc{\WE} &\sc{\SE} &\sc{\SE} &\sc{\SE}\\
\sc{\NS} &\sc{\SW} &\sc{\WE} &\sc{\SE}\\
\sc{\WE} &\sc{\SE} &\sc{\NE} &\sc{\SE}\\
\sc{\NE} &\sc{\SE} &\sc{\NE} &\sc{\SE}\\
\sc{\NS} &\sc{\SW} &\sc{\NW} &\sc{\WE}\\
\sc{\WE} &\sc{\SE} &\sc{\NE} &\sc{\NE}\\
\sc{\NS} &\sc{\WE} &\sc{\NE} &\sc{\NE}\\
\sc{\WE} &\sc{\NE} &\sc{\NE} &\sc{\NE}\\
\end{array}\right]
\!\!%
\begin{array}{c} 
\vcenter{\hbox{
\begin{tikzpicture}[x={(0in,-0.2in)},y={(0.2in,0in)}] 
\foreach \i in {1,...,9}
\foreach \j in {1,...,4}
\draw(\i,\j)node{$\bullet$};
\foreach \j in {1,...,4}
\draw(8,\j)node{$\bullet$};
\draw(1,-5)node{$1$};
\draw(2,-5)node{$2$};
\draw(3,-5)node{$3$};
\draw(4,-5)node{$4$};
\draw(5,-5)node{$0$}; 
\draw(6,-5)node{$\ov4$};
\draw(7,-5)node{$\ov3$};
\draw(8,-5)node{$\ov2$};
\draw(9,-5)node{$\ov1$};
\draw(5,0)node{$\bullet$};
\draw[ultra thick] (1,1)to(1,1.25); \draw[ultra thick, <-] (1,1.25)to(1,2);
\draw[ultra thick] (1,2)to(1,2.25); \draw[ultra thick,<-] (1,2.25)to(1,3);
\draw[ultra thick] (1,3)to(1,3.25); \draw[ultra thick,<-] (1,3.25)to(1,4);
\draw[ultra thick] (1,4)to(1,4.25); \draw[ultra thick,<-] (1,4.25)to(1,5);
\draw[ultra thick] (2,1)to(2,1.25);  \draw[ ultra thick,<-] (2,1.25)to(2,2);
\draw[ultra thick] (2,2)to(2,2.25);  \draw[ultra thick,<-] (2,2.25)to(2,3);
\draw[ultra thick] (2,3)to(2,3.25);  \draw[ ultra thick,<-] (2,3.25)to(2,4);
\draw[ultra thick] (2,4)to(2,4.25); \draw[ultra thick,<-] (2,4.25)to(2,5);
\draw[ultra thick, ->] (3,1)to(3,1.75);  \draw[ ultra thick] (3,1.75)to(3,2);
\draw[ultra thick, ->] (3,2)to(3,2.75);  \draw[ ultra thick] (3,2.75)to(3,3);
\draw[,ultra thick] (3,3)to(3,3.25);  \draw[ultra thick,<-] (3,3.25)to(3,4);
\draw[ultra thick] (3,4)to(3,4.25); \draw[ultra thick,<-] (3,4.25)to(3,5);
\draw[ultra thick] (4,1)to(4,1.25);  \draw[ ultra thick,<-] (4,1.25)to(4,2);
\draw[ultra thick] (4,2)to(4,2.25);  \draw[ ultra thick,<-] (4,2.25)to(4,3);
\draw[ultra thick] (4,3)to(4,3.25);  \draw[ ultra thick,<-] (4,3.25)to(4,4);
\draw[ultra thick] (4,4)to(4,4.25); \draw[ultra thick,<-] (4,4.25)to(4,5);
\draw[ultra thick] (5,1)to(5,1.25);  \draw[ultra thick,<-] (5,1.25)to(5,2);
\draw[ultra thick] (5,2)to(5,2.25);  \draw[ ultra thick,<-] (5,2.25)to(5,3);
\draw[ultra thick] (5,3)to(5,3.25);  \draw[ ultra thick,<-] (5,3.25)to(5,4);
\draw[ultra thick] (5,4)to(5,4.25); \draw[ultra thick,<-] (5,4.25)to(5,5);
 \draw[ultra thick,<-] (5,0.25)to(5,1);

\draw[ultra thick, ->] (6,1)to(6,1.75);  \draw[ultra thick] (6,1.75)to(6,2);
\draw[ultra thick, ->] (6,2)to(6,2.75);  \draw[ ultra thick] (6,2.75)to(6,3);
\draw[ultra thick, ->] (6,3)to(6,3.75);  \draw[ ultra thick] (6,3.75)to(6,4);
\draw[ultra thick] (6,4)to(6,4.25); \draw[ultra thick,<-] (6,4.25)to(6,5);

\draw[ultra thick] (7,1)to(7,1.25);  \draw[ ultra thick,<-] (7,1.25)to(7,2);
\draw[ultra thick] (7,2)to(7,2.25);  \draw[ ultra thick,<-] (7,2.25)to(7,3);
\draw[ultra thick] (7,3)to(7,3.25);  \draw[ ultra thick,<-] (7,3.25)to(7,4);
\draw[ultra thick] (7,4)to(7,4.25); \draw[ultra thick,<-] (7,4.25)to(7,5);

\draw[ultra thick, ->] (8,1)to(8,1.75);  \draw[ ultra thick] (8,1.75)to(8,2);
\draw[ultra thick] (8,2)to(8,2.25);  \draw[ ultra thick,<-] (8,2.25)to(8,3);
\draw[ultra thick] (8,3)to(8,3.25);  \draw[ ultra thick,<-] (8,3.25)to(8,4);
\draw[ultra thick] (8,4)to(8,4.25); \draw[ultra thick,<-] (8,4.25)to(8,5);

\draw[ultra thick] (9,1)to(9,1.25);  \draw[ ultra thick,<-] (9,1.25)to(9,2);
\draw[ultra thick] (9,2)to(9,2.25);  \draw[ ultra thick,<-] (9,2.25)to(9,3);
\draw[ultra thick] (9,3)to(9,3.25);  \draw[ ultra thick,<-] (9,3.25)to(9,4);
\draw[ultra thick] (9,4)to(9,4.25); \draw[ultra thick,<-] (9,4.25)to(9,5);

%
\draw[ultra thick](1,1)to(1.25,1); \draw[ultra thick,<-] (1.25,1)to(2,1);
\draw[ ultra thick](1,2)to(1.25,2); \draw[ultra thick,<-] (1.25,2)to(2,2);
\draw[ ultra thick](1,3)to(1.25,3); \draw[ ultra thick,<-] (1.25,3)to(2,3);
\draw[ ultra thick](1,4)to(1.25,4); \draw[ ultra thick,<-] (1.25,4)to(2,4);
\draw[ultra thick](0,1)to(0.25,1); \draw[ultra thick,<-] (0.25,1)to(1,1);
\draw[ultra thick](0,2)to(0.25,2); \draw[ultra thick,<-] (0.25,2)to(1,2);
\draw[ultra thick](0,3)to(0.25,3); \draw[ultra thick,<-] (0.25,3)to(1,3);
\draw[ultra thick](0,4)to(0.25,4); \draw[ultra thick,<-] (0.25,4)to(1,4);

\draw[ ultra thick,->](2,1)to(2.75,1); \draw[ultra thick] (2.75,1)to(3,1);
\draw[ ultra thick](2,2)to(2.25,2); \draw[ultra thick,<-] (2.25,2)to(3,2);
\draw[ ultra thick](2,3)to(2.25,3); \draw[ ultra thick,<-] (2.25,3)to(3,3);
\draw[ ultra thick](2,4)to(2.25,4); \draw[ ultra thick,<-] (2.25,4)to(3,4);
\draw[ ultra thick](3,1)to(3.25,1); \draw[ultra thick,<-] (3.25,1)to(4,1);
\draw[ultra thick](3,2)to(3.25,2); \draw[ultra thick,<-] (3.25,2)to(4,2);
\draw[ ultra thick,->](3,3)to(3.75,3); \draw[ ultra thick] (3.75,3)to(4,3);
\draw[ ultra thick](3,4)to(3.25,4); \draw[ ultra thick,<-] (3.25,4)to(4,4);
\draw[ultra thick,->](4,1)to(4.75,1); \draw[ ultra thick] (4.75,1)to(5,1);
\draw[ ultra thick](4,2)to(4.25,2); \draw[ultra thick,<-] (4.25,2)to(5,2);
\draw[ ultra thick,->](4,3)to(4.75,3); \draw[ ultra thick] (4.75,3)to(5,3);
\draw[ ultra thick](4,4)to(4.25,4); \draw[ ultra thick,<-] (4.25,4)to(5,4);
\draw[ ultra thick,->](5,1)to(5.75,1); \draw[ ultra thick] (5.75,1)to(6,1);
\draw[ultra thick](5,2)to(5.25,2); \draw[ ultra thick,<-] (5.25,2)to(6,2);
\draw[ ultra thick,->](5,3)to(5.75,3); \draw[ ultra thick] (5.75,3)to(6,3);
\draw[ ultra thick](5,4)to(5.25,4); \draw[ ultra thick,<-] (5.25,4)to(6,4);
\draw[ ultra thick](6,1)to(6.25,1); \draw[ultra thick,<-] (6.25,1)to(7,1);
\draw[ ultra thick](6,2)to(6.25,2); \draw[ ultra thick,<-] (6.25,2)to(7,2);
\draw[ ultra thick,->](6,3)to(6.75,3); \draw[ultra thick] (6.75,3)to(7,3);
\draw[ ultra thick,->](6,4)to(6.75,4); \draw[ ultra thick] (6.75,4)to(7,4);
\draw[ultra thick,->](7,1)to(7.75,1); \draw[ultra thick] (7.75,1)to(8,1);
\draw[ ultra thick](7,2)to(7.25,2); \draw[ ultra thick,<-] (7.25,2)to(8,2);
\draw[ultra thick,->](7,3)to(7.75,3); \draw[ ultra thick] (7.75,3)to(8,3);
\draw[ ultra thick,->](7,4)to(7.75,4); \draw[ ultra thick] (7.75,4)to(8,4);
\draw[ultra thick](8,1)to(8.25,1); \draw[ultra thick,<-] (8.25,1)to(9,1);
\draw[ ultra thick,->](8,2)to(8.75,2); \draw[ ultra thick] (8.75,2)to(9,2);
\draw[ultra thick,->](8,3)to(8.75,3); \draw[ ultra thick] (8.75,3)to(9,3);
\draw[ ultra thick,->](8,4)to(8.75,4); \draw[ ultra thick] (8.75,4)to(9,4);
\draw[ultra thick,->](9,1)to(9.75,1); \draw[ultra thick] (9.75,1)to(10,1);
\draw[ ultra thick,->](9,2)to(9.75,2); \draw[ ultra thick] (9.75,2)to(10,2);
\draw[ultra thick,->](9,3)to(9.75,3); \draw[ ultra thick] (9.75,3)to(10,3);
\draw[ ultra thick,->](9,4)to(9.75,4); \draw[ ultra thick] (9.75,4)to(10,4);
%
\draw[ultra thick] (1,0)to(1,0.25); \draw[ultra thick, <-] (1,0.25)to(1,1);
\draw[ultra thick, ->] (2,0)to(2,0.75); \draw[ultra thick] (2,0.75)to(2,1);
\draw[ultra thick] (3,0)to(3,0.25); \draw[ultra thick, <-] (3,0.25)to(3,1);
\draw[ultra thick, ->] (4,0)to(4,0.75); \draw[ultra thick] (4,0.75)to(4,1);
\draw[ultra thick] (5,0)to(5,0.25); \draw[ultra thick, <-] (5,0.25)to(5,1);
\draw[ultra thick] (6,0)to(6,0.25); \draw[ultra thick, <-] (6,0.25)to(6,1);
\draw[ultra thick, ->] (7,0)to(7,0.75); \draw[ultra thick] (7,0.75)to(7,1);
\draw[ultra thick] (8,0)to(8,0.25); \draw[ultra thick, <-] (8,0.25)to(8,1);
\draw[ultra thick, ->] (9,0)to(9,0.75); \draw[ultra thick] (9,0.75)to(9,1);
\draw[ultra thick,out= 90,in=180] (5,-4) to (1,0);                   
\draw[ultra thick,out= 270,in=180] (5,-4) to (9,0);                  
\draw[ultra thick,out= 90,in=180] (5,-3) to (2,0);                  
\draw[ultra thick,out= 270,in=180] (5,-3) to (8,0);                 
\draw[ultra thick,out= 90,in=180] (5,-2) to (3,0);                  
\draw[ultra thick,out= 270,in=180] (5,-2) to (7,0);                
\draw[ultra thick,out= 90,in=180] (5,-1) to (4,0);
\draw[ultra thick,out= 270,in=180] (5,-1) to (6,0);
\end{tikzpicture}
}}
\end{array}
\end{array}
\end{equation*}
\end{center}
\caption{Example of a half HTSASM, its associated CPM, and the corresponding SIC
in the case $N=2n+1$ with $n=4$.}
\label{fig-ACS-delta}
\end{figure*}

Before assigning weights to these combinatorial objects it is useful to extend their definition.
Let ${\cal P}$ denote the set of all partitions $\lambda=(\lambda_1,\lambda_2,\ldots)$ with
weakly decreasing non-negative integer parts, and let ${\cal D}$ be the subset consisting of strict partitions 
whose parts are strictly decreasing positive integers. In both cases the partition $\lambda$ is said
to have weight $|\lambda|$ equal to the sum of its parts and length $\ell(\lambda)$ equal to the number
of its non-zero parts.


For our purposes in which we deal with the right hand portions 
of HTSASMs we require the following, which is the half-turn version of a definition due to 
Okada~\cite{Okada}:
\begin{Definition} For fixed $N,m\in\N$, with $N=2n$ or $2n+1$ and $m \geq n$, and strict partition $\lambda\in{\cal D}$ of 
length $\ell(\lambda)=n$ and largest part $\lambda_1\leq m$, an $N\times m$ matrix $A$ is said to be a $\lambda$-HTSASM if the entries $a_{ij}$ for $1\leq i\leq N$ and $1\leq j\leq m$ satisfy the conditions:
\begin{enumerate}
\item $a_{ij}\in\{1,-1,0\}$;
\item $\sum_{k=j}^m a_{ik}\in\{0,1\}$, $\sum_{k=1}^m a_{ik}+\sum_{k=1}^j a_{N+1-i,k}\in\{0,1\}$,   and $\sum_{k=1}^i a_{kj}\in\{0,1\}$;
\item $\sum_{j=1}^m (a_{ij}+a_{N+1-i,j})=1$ and $\sum_{i=1}^{N} a_{ij}=1$ if $j=\lambda_k$ for some $k$ and is $0$ otherwise.
\end{enumerate}
\end{Definition}

A typical $\lambda$-HTSASM is illustrated below in Figure~\ref{fig-ACS-lambda} 
in the case $N=2n+1$ with $n=3$ and $\lambda=(8,6,3)$, along with its associated CPM and U-turn 
square ice configuration. 

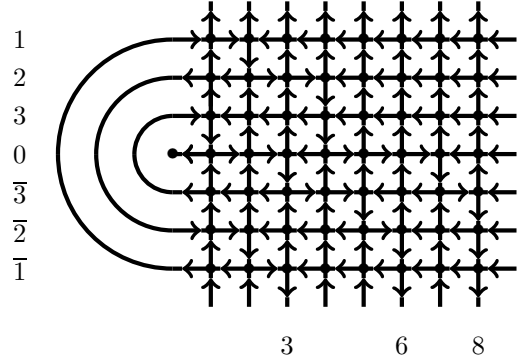
\begin{figure*}[htbp]
\begin{center}
\begin{equation*}\hskip-1ex
\begin{array}{c}
  \begin{array}{c}
   \begin{array}{cccccccc}
    1&2&3&4&5&6&7&8\\
   \end{array}\cr
   \left[
     \begin{array}{cccccccc}
     0 & 1 & 0 & 0 & 0 & 0 & 0 & 0\\
     0 &\ov1& 0 & 1 & 0 & 0 & 0 & 0\\
     1 & 0 & 0 & 0 & 0 & 0 & 0 & 0\\
   \ov1& 0 & 1 & \ov1 & 0 & 0 & 1 & 0\\
     0 & 0 & \ov1 & 0 & 1 & 0 & \ov1 & 1\\
     0 & 1 & 0 & 0 & \ov1 & 1 & 0 & 0\\
     0 & \ov1 & 1 & 0 & 0 & 0 & 0 & 0\\
     \end{array}\right]\cr
   \begin{array}{cccccccc}
    0&0&1&0&0&1&0&1\\
   \end{array}
	\end{array}
\!\!%
\left[\begin{array}{cccccccc}
\sc{\SW} &\sc{\WE} &\sc{\SE} &\sc{\SE} &\sc{\SE} &\sc{\SE} &\sc{\SE} &\sc{\SE}\\
\sc{\SE} &\sc{\NS} &\sc{\SW} &\sc{\WE} &\sc{\SE} &\sc{\SE} &\sc{\SE} &\sc{\SE}\\
\sc{\WE} &\sc{\SE} &\sc{\SE} &\sc{\NE} &\sc{\SE} &\sc{\SE} &\sc{\SE} &\sc{\SE}\\
\sc{\NS} &\sc{\SW} &\sc{\WE} &\sc{\NS} &\sc{\SW} &\sc{\SW} &\sc{\WE} &\sc{\SE}\\
\sc{\SE} &\sc{\SE} &\sc{\NS} &\sc{\SW} &\sc{\WE} &\sc{\SE} &\sc{\NS} &\sc{\WE}\\
\sc{\SW} &\sc{\WE} &\sc{\SE} &\sc{\SE} &\sc{\NS} &\sc{\WE} &\sc{\SE} &\sc{\NE}\\
\sc{\SE} &\sc{\NS} &\sc{\WE} &\sc{\SE} &\sc{\SE} &\sc{\NE} &\sc{\SE} &\sc{\NE}\\
\end{array}\right]
\cr\cr 
\vcenter{\hbox{
\begin{tikzpicture}[x={(0in,-0.2in)},y={(0.2in,0in)}] 
\foreach \i in {1,...,7}
\foreach \j in {1,...,8}
\draw(\i,\j)node{$\bullet$};
\draw(4,0)node{$\bullet$};
\draw(1,-4)node{$1$};
\draw(2,-4)node{$2$};
\draw(3,-4)node{$3$};
\draw(4,-4)node{$0$}; 
\draw(5,-4)node{$\ov3$};
\draw(6,-4)node{$\ov2$};
\draw(7,-4)node{$\ov1$};
\draw(9,3)node{$3$};
\draw(9,6)node{$6$};
\draw(9,8)node{$8$};
%
\draw[ultra thick, ->] (1,1)to(1,1.75); \draw[ultra thick] (1,1.75)to(1,2);
\draw[ultra thick] (1,2)to(1,2.25); \draw[ultra thick,<-] (1,2.25)to(1,3);
\draw[ultra thick] (1,3)to(1,3.25); \draw[ultra thick,<-] (1,3.25)to(1,4);
\draw[ultra thick] (1,4)to(1,4.25); \draw[ultra thick,<-] (1,4.25)to(1,5);
\draw[ultra thick] (1,5)to(1,5.25); \draw[ultra thick,<-] (1,5.25)to(1,6);
\draw[ultra thick] (1,6)to(1,6.25); \draw[ultra thick,<-] (1,6.25)to(1,7);
\draw[ultra thick] (1,7)to(1,7.25); \draw[ultra thick,<-] (1,7.25)to(1,8);
\draw[ultra thick] (1,8)to(1,8.25); \draw[ultra thick,<-] (1,8.25)to(1,9);
\draw[ultra thick] (2,1)to(2,1.25);  \draw[ ultra thick,<-] (2,1.25)to(2,2);
\draw[ultra thick,->] (2,2)to(2,2.75);  \draw[ultra thick] (2,2.75)to(2,3);
\draw[ultra thick,->] (2,3)to(2,3.75);  \draw[ ultra thick] (2,3.75)to(2,4);
\draw[ultra thick] (2,4)to(2,4.25); \draw[ultra thick,<-] (2,4.25)to(2,5);
\draw[ultra thick] (2,5)to(2,5.25);  \draw[ ultra thick,<-] (2,5.25)to(2,6);
\draw[ultra thick] (2,6)to(2,6.25);  \draw[ultra thick,<-] (2,6.25)to(2,7);
\draw[ultra thick] (2,7)to(2,7.25);  \draw[ ultra thick,<-] (2,7.25)to(2,8);
\draw[ultra thick] (2,8)to(2,8.25); \draw[ultra thick,<-] (2,8.25)to(2,9);
\draw[ultra thick] (3,1)to(3,1.25);  \draw[ultra thick,<-] (3,1.25)to(3,2);
\draw[ultra thick] (3,2)to(3,2.25);  \draw[ultra thick,<-] (3,2.25)to(3,3);
\draw[ultra thick] (3,3)to(3,3.25);  \draw[ultra thick,<-] (3,3.25)to(3,4);
\draw[ultra thick] (3,4)to(3,4.25);  \draw[ultra thick,<-] (3,4.25)to(3,5);
\draw[ultra thick] (3,5)to(3,5.25);  \draw[ultra thick,<-] (3,5.25)to(3,6);
\draw[ultra thick] (3,6)to(3,6.25);  \draw[ ultra thick,<-] (3,6.25)to(3,7);
\draw[ultra thick] (3,7)to(3,7.25);  \draw[ultra thick,<-] (3,7.25)to(3,8);
\draw[ultra thick] (3,8)to(3,8.25); \draw[ultra thick,<-] (3,8.25)to(3,9);
\draw[ultra thick,<-] (4,0.25)to(4,1);
\draw[ultra thick,->] (4,1)to(4,1.75);  \draw[ ultra thick] (4,1.75)to(4,2);
\draw[ultra thick,->] (4,2)to(4,2.75);  \draw[ ultra thick] (4,2.75)to(4,3);
\draw[ultra thick] (4,3)to(4,3.25);  \draw[ ultra thick,<-] (4,3.25)to(4,4);
\draw[ultra thick,->] (4,4)to(4,4.75); \draw[ultra thick] (4,4.75)to(4,5);
\draw[ultra thick,->] (4,5)to(4,5.75);  \draw[ ultra thick] (4,5.75)to(4,6);
\draw[ultra thick,->] (4,6)to(4,6.75);  \draw[ ultra thick] (4,6.75)to(4,7);
\draw[ultra thick] (4,7)to(4,7.25);  \draw[ ultra thick,<-] (4,7.25)to(4,8);
\draw[ultra thick] (4,8)to(4,8.25); \draw[ultra thick,<-] (4,8.25)to(4,9);
\draw[ultra thick] (5,1)to(5,1.25);  \draw[ultra thick,<-] (5,1.25)to(5,2);
\draw[ultra thick] (5,2)to(5,2.25);  \draw[ ultra thick,<-] (5,2.25)to(5,3);
\draw[ultra thick,->] (5,3)to(5,3.75);  \draw[ ultra thick] (5,3.75)to(5,4);
\draw[ultra thick,->] (5,4)to(5,4.75); \draw[ultra thick] (5,4.75)to(5,5);
\draw[ultra thick] (5,5)to(5,5.25);  \draw[ultra thick,<-] (5,5.25)to(5,6);
\draw[ultra thick] (5,6)to(5,6.25);  \draw[ ultra thick,<-] (5,6.25)to(5,7);
\draw[ultra thick,->] (5,7)to(5,7.75);  \draw[ ultra thick] (5,7.75)to(5,8);
\draw[ultra thick] (5,8)to(5,8.25); \draw[ultra thick,<-] (5,8.25)to(5,9);
\draw[ultra thick, ->] (6,1)to(6,1.75);  \draw[ultra thick] (6,1.75)to(6,2);
\draw[ultra thick] (6,2)to(6,2.25);  \draw[ ultra thick,<-] (6,2.25)to(6,3);
\draw[ultra thick] (6,3)to(6,3.25);  \draw[ ultra thick,<-] (6,3.25)to(6,4);
\draw[ultra thick] (6,4)to(6,4.25); \draw[ultra thick,<-] (6,4.25)to(6,5);
\draw[ultra thick, ->] (6,5)to(6,5.75);  \draw[ultra thick] (6,5.75)to(6,6);
\draw[ultra thick] (6,6)to(6,6.25);  \draw[ ultra thick,<-] (6,6.25)to(6,7);
\draw[ultra thick] (6,7)to(6,7.25);  \draw[ ultra thick,<-] (6,7.25)to(6,8);
\draw[ultra thick] (6,8)to(6,8.25); \draw[ultra thick,<-] (6,8.25)to(6,9);
\draw[ultra thick] (7,1)to(7,1.25);  \draw[ ultra thick,<-] (7,1.25)to(7,2);
\draw[ultra thick,->] (7,2)to(7,2.75);  \draw[ ultra thick] (7,2.75)to(7,3);
\draw[ultra thick] (7,3)to(7,3.25);  \draw[ ultra thick,<-] (7,3.25)to(7,4);
\draw[ultra thick] (7,4)to(7,4.25); \draw[ultra thick,<-] (7,4.25)to(7,5);
\draw[ultra thick] (7,5)to(7,5.25);  \draw[ ultra thick,<-] (7,5.25)to(7,6);
\draw[ultra thick] (7,6)to(7,6.25);  \draw[ ultra thick,<-] (7,6.25)to(7,7);
\draw[ultra thick] (7,7)to(7,7.25);  \draw[ ultra thick,<-] (7,7.25)to(7,8);
\draw[ultra thick] (7,8)to(7,8.25); \draw[ultra thick,<-] (7,8.25)to(7,9);
%
%
\draw[ultra thick](0,1)to(0.25,1); \draw[ultra thick,<-] (0.25,1)to(1,1);
\draw[ultra thick](0,2)to(0.25,2); \draw[ultra thick,<-] (0.25,2)to(1,2);
\draw[ultra thick](0,3)to(0.25,3); \draw[ultra thick,<-] (0.25,3)to(1,3);
\draw[ultra thick](0,4)to(0.25,4); \draw[ultra thick,<-] (0.25,4)to(1,4);
\draw[ultra thick](0,5)to(0.25,5); \draw[ultra thick,<-] (0.25,5)to(1,5);
\draw[ultra thick](0,6)to(0.25,6); \draw[ultra thick,<-] (0.25,6)to(1,6);
\draw[ultra thick](0,7)to(0.25,7); \draw[ultra thick,<-] (0.25,7)to(1,7);
\draw[ultra thick](0,8)to(0.25,8); \draw[ultra thick,<-] (0.25,8)to(1,8);
\draw[ultra thick](1,1)to(1.25,1); \draw[ultra thick,<-] (1.25,1)to(2,1);
\draw[ultra thick,->](1,2)to(1.75,2); \draw[ultra thick] (1.75,2)to(2,2);
\draw[ultra thick](1,3)to(1.25,3); \draw[ultra thick,<-] (1.25,3)to(2,3);
\draw[ultra thick](1,4)to(1.25,4); \draw[ultra thick,<-] (1.25,4)to(2,4);
\draw[ultra thick](1,5)to(1.25,5); \draw[ultra thick,<-] (1.25,5)to(2,5);
\draw[ultra thick](1,6)to(1.25,6); \draw[ultra thick,<-] (1.25,6)to(2,6);
\draw[ultra thick](1,7)to(1.25,7); \draw[ultra thick,<-] (1.25,7)to(2,7);
\draw[ultra thick](1,8)to(1.25,8); \draw[ultra thick,<-] (1.25,8)to(2,8);
\draw[ ultra thick](2,1)to(2.75,1); \draw[ultra thick,<-] (2.25,1)to(3,1);
\draw[ ultra thick](2,2)to(2.25,2); \draw[ultra thick,<-] (2.25,2)to(3,2);
\draw[ ultra thick](2,3)to(2.25,3); \draw[ ultra thick,<-] (2.25,3)to(3,3);
\draw[ ultra thick,->](2,4)to(2.75,4); \draw[ ultra thick] (2.75,4)to(3,4);
\draw[ ultra thick](2,5)to(2.25,5); \draw[ultra thick,<-] (2.25,5)to(3,5);
\draw[ ultra thick](2,6)to(2.25,6); \draw[ultra thick,<-] (2.25,6)to(3,6);
\draw[ ultra thick](2,7)to(2.25,7); \draw[ ultra thick,<-] (2.25,7)to(3,7);
\draw[ ultra thick](2,8)to(2.25,8); \draw[ ultra thick,<-] (2.25,8)to(3,8);
\draw[ ultra thick,->](3,1)to(3.75,1); \draw[ultra thick] (3.75,1)to(4,1);
\draw[ultra thick](3,2)to(3.25,2); \draw[ultra thick,<-] (3.25,2)to(4,2);
\draw[ ultra thick](3,3)to(3.25,3); \draw[ ultra thick,<-] (3.25,3)to(4,3);
\draw[ ultra thick,->](3,4)to(3.75,4); \draw[ ultra thick] (3.75,4)to(4,4);
\draw[ ultra thick](3,5)to(3.25,5); \draw[ultra thick,<-] (3.25,5)to(4,5);
\draw[ultra thick](3,6)to(3.25,6); \draw[ultra thick,<-] (3.25,6)to(4,6);
\draw[ ultra thick](3,7)to(3.25,7); \draw[ ultra thick,<-] (3.25,7)to(4,7);
\draw[ ultra thick](3,8)to(3.25,8); \draw[ ultra thick,<-] (3.25,8)to(4,8);
\draw[ultra thick](4,1)to(4.25,1); \draw[ ultra thick,<-] (4.25,1)to(5,1);
\draw[ ultra thick](4,2)to(4.25,2); \draw[ultra thick,<-] (4.25,2)to(5,2);
\draw[ ultra thick,->](4,3)to(4.75,3); \draw[ ultra thick] (4.75,3)to(5,3);
\draw[ ultra thick](4,4)to(4.25,4); \draw[ ultra thick,<-] (4.25,4)to(5,4);
\draw[ultra thick](4,5)to(4.25,5); \draw[ ultra thick,<-] (4.25,5)to(5,5);
\draw[ ultra thick](4,6)to(4.25,6); \draw[ultra thick,<-] (4.25,6)to(5,6);
\draw[ ultra thick,->](4,7)to(4.75,7); \draw[ ultra thick] (4.75,7)to(5,7);
\draw[ ultra thick](4,8)to(4.25,8); \draw[ ultra thick,<-] (4.25,8)to(5,8);
\draw[ ultra thick](5,1)to(5.25,1); \draw[ ultra thick,<-] (5.25,1)to(6,1);
\draw[ultra thick](5,2)to(5.25,2); \draw[ ultra thick,<-] (5.25,2)to(6,2);
\draw[ ultra thick](5,3)to(5.25,3); \draw[ ultra thick,<-] (5.25,3)to(6,3);
\draw[ ultra thick](5,4)to(5.25,4); \draw[ ultra thick,<-] (5.25,4)to(6,4);
\draw[ ultra thick,->](5,5)to(5.75,5); \draw[ ultra thick] (5.75,5)to(6,5);
\draw[ultra thick](5,6)to(5.25,6); \draw[ ultra thick,<-] (5.25,6)to(6,6);
\draw[ ultra thick](5,7)to(5.25,7); \draw[ ultra thick,<-] (5.25,7)to(6,7);
\draw[ ultra thick,->](5,8)to(5.75,8); \draw[ ultra thick] (5.75,8)to(6,8);
\draw[ ultra thick](6,1)to(6.25,1); \draw[ultra thick,<-] (6.25,1)to(7,1);
\draw[ ultra thick,->](6,2)to(6.75,2); \draw[ ultra thick] (6.75,2)to(7,2);
\draw[ ultra thick](6,3)to(6.25,3); \draw[ultra thick,<-] (6.25,3)to(7,3);
\draw[ ultra thick](6,4)to(6.25,4); \draw[ ultra thick,<-] (6.25,4)to(7,4);
\draw[ ultra thick](6,5)to(6.25,5); \draw[ultra thick,<-] (6.25,5)to(7,5);
\draw[ ultra thick,->](6,6)to(6.75,6); \draw[ ultra thick] (6.75,6)to(7,6);
\draw[ ultra thick](6,7)to(6.25,7); \draw[ultra thick,<-] (6.25,7)to(7,7);
\draw[ ultra thick,->](6,8)to(6.75,8); \draw[ ultra thick] (6.75,8)to(7,8);
\draw[ultra thick](7,1)to(7.25,1); \draw[ultra thick,<-] (7.25,1)to(8,1);
\draw[ ultra thick](7,2)to(7.25,2); \draw[ ultra thick,<-] (7.25,2)to(8,2);
\draw[ultra thick,->](7,3)to(7.75,3); \draw[ ultra thick] (7.75,3)to(8,3);
\draw[ ultra thick](7,4)to(7.25,4); \draw[ ultra thick,<-] (7.25,4)to(8,4);
\draw[ultra thick](7,5)to(7.25,5); \draw[ultra thick,<-] (7.25,5)to(8,5);
\draw[ ultra thick,->](7,6)to(7.75,6); \draw[ ultra thick] (7.75,6)to(8,6);
\draw[ultra thick](7,7)to(7.25,7); \draw[ ultra thick,<-]  (7.25,7)to(8,7);
\draw[ ultra thick,->](7,8)to(7.75,8); \draw[ ultra thick] (7.75,8)to(8,8);
%
\draw[ultra thick, ->] (1,0)to(1,0.75); \draw[ultra thick] (1,0.75)to(1,1);
\draw[ultra thick] (2,0)to(2,0.25); \draw[ultra thick, <-] (2,0.25)to(2,1);
\draw[ultra thick, ->] (3,0)to(3,0.75); \draw[ultra thick] (3,0.75)to(3,1);
\draw[ultra thick] (4,0)to(4,0.25); \draw[ultra thick, <-] (4,0.25)to(4,1);
\draw[ultra thick] (5,0)to(5,0.25); \draw[ultra thick, <-] (5,0.25)to(5,1);
\draw[ultra thick, ->] (6,0)to(6,0.75); \draw[ultra thick] (6,0.75)to(6,1);
\draw[ultra thick] (7,0)to(7,0.25); \draw[ultra thick, <-] (7,0.25)to(7,1);
\draw[ultra thick,out= 90,in=180] (4,-3) to (1,0);                   
\draw[ultra thick,out= 270,in=180] (4,-3) to (7,0);                  
\draw[ultra thick,out= 90,in=180] (4,-2) to (2,0);                  
\draw[ultra thick,out= 270,in=180] (4,-2) to (6,0);                 
\draw[ultra thick,out= 90,in=180] (4,-1) to (3,0);                  
\draw[,ultra thick,out= 270,in=180] (4,-1) to (5,0);                
\end{tikzpicture}
}}
%
\end{array}
\end{equation*}
\end{center}
\caption{Example of a $\lambda$-HTSASM, its associated CPM, and the corresponding U-turn SIC 
in the case $N=2n+1$ with $n=3$ and $\lambda=(8,6,3)$.}
\label{fig-ACS-lambda}
\end{figure*}

The column sums of the $\lambda$-HTSASM are $1$ for those columns labelled by the
parts $8$, $6$ and $3$ of $\lambda$, and $0$ otherwise. By the same token, while the outer edge
directions of the U-turn SIC are all upwards on the upper boundary and leftwards on the right hand boundary
and on the left hand end of the central row, those on the lower boundary are downwards only 
in the columns labelled by the parts $8$, $6$ and $3$ of $\lambda$, and upwards otherwise.   

\section{Weighted \texorpdfstring{$\lambda$}{lambda}-HTSASMs}
\label{sec:ASMweights}

The connection between alternating sign matrices and compass point matrices is such that
a number of remarkable identities involving various statistics on ASMs can be recast
in terms of statistics on the corresponding CPMs.
First we have the following two theorems in the special case of $\delta$-HTSASMs:

\begin{Theorem}[Okada~\cite{Okada}]\label{the-Okada}
For $n\in\N$ let $\delta=(n,n-1,\ldots,1)$ and let ${\cal B}^\delta_n$ be the set of all $2n\times n$ 
$\delta$-HTSASMs. For each $A\in{\cal B}^\delta_n$ let $C$ be the corresponding $2n\times n$
CPM with matrix elements $c_{ij}$. Let $\x=(x_1,x_2,\ldots,x_{n})$ be a sequence of non-zero indeterminates. 
Then for any $t$ we have
\begin{equation}
      \sum_{A\in{\cal B}^\delta_n} \ \wgt(A) = 
			\prod_{i=1}^n (1-tx_i)\ \prod_{1\leq i<j\leq n} (1-t^2x_ix_j)(1-t^2x_ix_j^{-1})\,,
\end{equation} 
where 
\begin{equation}
   \wgt(A) = \prod_{i=1}^n  x_i^{-i}\    \prod_{i=1}^{2n}\prod_{j=1}^n  \wgt(c_{ij})\,,
\end{equation}
with $\wgt(c_{ij})$ as tabulated below: 
\begin{equation}\label{tab-Okada}
\begin{array}{|l|l|l|}
\hline
\hbox{Entry}&i\leq n&i>n\cr
\hbox{at $(i,j)$}&j\leq n&j\leq n\cr
\hline
\WE&\sqrt{-1}\,x_i&x_{2n+1-i}\cr
\NS&-\sqrt{-1}\,(1-t^2)&(1-t^2)\cr
\NE&\sqrt{-1}\,t&\sqrt{-1}\,t\,x_{2n+1-i}\cr
\SE&1&x_{2n+1-i}\cr
\NW&x_i&1\cr
\SW&\sqrt{-1}\,t\,x_i&\sqrt{-1}\,t\cr
\hline
\end{array}
\end{equation}
\end{Theorem}

\begin{Theorem}[Simpson~\cite{Simpson}]\label{the-Simpson}
For $n\in\N$ let $\delta=(n,n-1,\ldots,1)$ and let ${\cal B'}^\delta_n$ be the set of all $(2n+1)\times n$ 
$\delta$-HTSASMs. For each $A\in{\cal B'}^\delta_n$ let $C$ be the corresponding $(2n+1)\times n$
CPM with matrix elements $c_{ij}$. Let $\x=(x_1,x_2,\ldots,x_{n})$ be a sequence of non-zero indeterminates. 
Then for any $t>0$ we have
\begin{equation}
      \sum_{A\in{\cal B'}^\delta_n} \ \wgt(A) = 
			\prod_{i=1}^n (1+t^2x_i)\ \prod_{1\leq i<j\leq n} (1+t^2x_ix_j)(1+t^2x_ix_j^{-1})\,,
\end{equation} 
where 
\begin{equation}
   \wgt(A) = \prod_{i=1}^n  x_i^{-i}\    \prod_{i=1}^{2n+1}\prod_{j=1}^n  \wgt(c_{ij})\,,
\end{equation}
with $\wgt(c_{ij})$ as tabulated below: 
\begin{equation}\label{tab-Simpson}
\begin{array}{|l|l|l|l|l|}
\hline
\hbox{Entry}&i<n+1&i<n+1&i=n+1&i>n+1\cr
\hbox{at $(i,j)$}&j=1&j>1&j\geq1&j\geq1\cr
\hline
\WE&x_i&t^{-1}\,x_i&1&x_{2n+2-i}\cr
\NS&t(1+t^2)&t(1+t^2)& (1+t^2)&\ds (1+t^2)\cr
\NE&t&t&t&t\,x_{2n+2-i}\cr
\SE&1&1&1&x_{2n+2-i}\cr
\NW&t\,x_i&x_i&1&1\cr
\SW&t^2x_i&t\,x_i&t&t\cr
\hline
\end{array}
\end{equation}
\end{Theorem}

The identities appearing in these two theorems can be generalised in two different ways: firstly 
by replacing the single parameter $t$ by a sequence of indeterminates, and then by 
passing from $\delta$ to an arbitrary strict partition $\lambda\in{\cal D}$ of length $n$
to obtain a factorisation formula of the Tokuyama type~\cite{Tokuyama} in the $\lambda$-HTSASM case. 
For example, it was found by Tabony~\cite{Tabony} that 
Okada's Theorem~\ref{the-Okada} can be generalised in both these ways. His result can be stated 
in the following way.

\begin{Theorem}[Tabony~\cite{Tabony}]\label{the-Tabony}
For $n\in\N$ let $\x=(x_1,x_2,\ldots,x_{n})$ be a sequence of non-zero indeterminates,
$\t=(t_1,t_2,\ldots,t_n)$ be a sequence of arbitrary parameters, and $\delta=(n,n-1,\ldots,1)$. 
For any $\lambda=\mu+\delta$, where $\mu$ is a partition of length $\ell(\mu)\leq n$
and $\lambda_1=m\geq n$, let ${\cal B}^\lambda_n$ be the set of all $2n\times m$ 
$\lambda$-HTSASMs. For each $A\in{\cal B}^\lambda_n$ let $C$ be the corresponding $2n\times m$
CPM with matrix elements $c_{ij}$ and let
\begin{equation}
   \wgt(A) =  \prod_{i=1}^n  x_i^{n-m-i}\    \prod_{i=1}^{2n}\prod_{j=1}^m  \wgt(c_{ij})\,,
\end{equation}
with $\wgt(c_{ij})$ as tabulated below: 
\begin{equation}\label{tab-Tabony}
\begin{array}{|l|l|l|}
\hline
\hbox{Entry}&i\leq n&i>n\cr
\hbox{at $(i,j)$}&j\leq n&j\leq n\cr
\hline
\WE&\sqrt{-1}\,x_i&x_{2n+1-i}\cr
\NS&-\sqrt{-1}\,(1-t_i^2)&(1-t_{2n+1-i}^2)\cr
\NE&\sqrt{-1}\,t_i&\sqrt{-1}\,t_{2n+1-i}x_{2n+1-i}\cr
\SE&1&x_{2n+1-i}\cr
\NW&x_i&1\cr
\SW&\sqrt{-1}\,t_i\,x_i&\sqrt{-1}\,t_{2n+1-i}\cr
\hline
\end{array}
\end{equation}

Then 
\begin{equation}\label{eqn-Tok-bn}
      \sum_{A\in{\cal B}^\lambda_n} \ \wgt(A) = \sum_{A\in{\cal B}^\delta_n} \  \wgt(A) \ \ \Phi_{B_n}^\mu(\x;\t)\,,
\end{equation}
where
\begin{equation}
			\sum_{A\in{\cal B}^\delta_n} \wgt(A) = 
			\prod_{i=1}^n (1-t_ix_i)\ \prod_{1\leq i<j\leq n} (1-t_it_jx_ix_j)(1-t_it_jx_ix_j^{-1})\,,
\end{equation}
and $\Phi_{B_n}^\mu(\x;\t)$ is symmetric under all permutations of $t_ix_i$ for $i=1,2,\ldots,n$ 
and any combination of inversions $x_i\mapsto 1/x_i$. 
\end{Theorem} 

In the same way, with the introduction of still more parameters as inspired by a suggestion of Brubaker~\cite{Brubaker} 
(see also~\cite{BrubakerSchultz}), we were led to conjecture the validity of the following
generalisation of Simpson's Theorem~\ref{the-Simpson}:
\begin{Theorem}\label{the-result1}
For $n\in\N$ let $\x=(x_1,x_2,\ldots,x_{n})$, $\y=(y_1,y_2,\ldots,y_{n})$, $\s=(s_1,s_2,\ldots,s_n)$
and $\t=(t_1,t_2,\ldots,t_n)$ be four sequences of non-zero parameters, 
and let $\z=(s_1x_1,\ldots,s_nx_n,z_0,t_1y_1^{-1},\ldots,t_ny_n^{-1})$, with $z_0$ a
further non-zero parameter. For $\delta=(n,n-1,\ldots,1)$
and any $\lambda=\mu+\delta$, where $\mu$ is a partition of length $\ell(\mu)\leq n$
and $\lambda_1=m\geq n$, let ${\cal B'}^\lambda_n$ be the set of all $(2n+1)\times m$ 
$\lambda$-HTSASMs. For each $A\in{\cal B'}^\lambda_n$ let $C$ be the corresponding $(2n+1)\times m$
CPM with matrix elements $c_{ij}$ and let
\begin{equation}
   \wgt(A) =  \prod_{i=1}^n  y_i^{n-m-i}\    \prod_{i=1}^{2n+1}\prod_{j=1}^m  \wgt(c_{ij})\,,
\end{equation}
with $\wgt(c_{ij})$ as tabulated below: 
\begin{equation}\label{tab-bnprime}
\begin{array}{|l|l|l|l|}
\hline
\hbox{Entry}&i<n+1&i=n&i>n+1\cr
\hbox{at $(i,j)$}&j\geq1&j\geq1&j\geq1\cr
\hline
\WE&y_i&1&y_{2n+2-i}\cr
\NS&1+s_ix_it_iy_i^{-1}&1+z_0^2&1+s_{2n+2-i}x_{2n+2-i}t_{2n+2-i}y^{-1}_{2n+2-i}\cr
\NE&t_i&u_0&s_{2n+2-i}x_{2n+2-i}\cr
\SE&1&1&y_{2n+2-i}\cr
\NW&y_i&1&1\cr
\SW&s_ix_i&z_0&t_{2n+2-i}\cr
\hline
\end{array}
\end{equation}

Then 
\begin{equation}\label{eqn-Tok-bnprime-Phi}
      \sum_{A\in{\cal B'}^\lambda_n} \ \wgt(A) = \sum_{A\in{\cal B'}^\delta_n} \  \wgt(A) \ \ \Phi_{B'_n}^\mu(\z)\,,
\end{equation}
where 
\begin{equation}
			\sum_{A\in{\cal B'}^\delta_n} \wgt(A) = 
			\prod_{i=1}^n (1+z_0s_ix_i)\ \prod_{1\leq i<j\leq n} (1+s_is_jx_ix_j)(1+s_it_jx_iy_j^{-1})\,,
\end{equation} 
and $\Phi_{B'_n}^\mu(\z)$ is symmetric under all permutations of the $2n+1$ components of $\z$.
\end{Theorem}

The precise identification of $\Phi_{B_n}^\mu(\z)$ and $\Phi_{B'_n}^\mu(\z)$ will be seen to emerge as 
a by-product of our proof of the validity of these theorems. In order to construct this proof it 
is necessary to introduce two further combinatorial entities, namely primed shifted tableaux (PSTs) 
and corresponding lattice path configurations (LPCs).
To prepare the ground for the use of these it is advantageous to rewrite the above statement of $\wgt(A)$ in terms
of entries $c_{ij}$ in an alternative form. 

First, for given $A$ and corresponding $(2n+1)\times m$ compass point matrix $C$ let $L_i$ be the total number 
of entries $\WE$, $\NW$ or $\SW$
that lie in the first column of $C$ above the entry $c_{i1}$ for $i=1,2,\ldots,2n+1$.
Then let $\#\XY_i$ be the number of pairs $\XY$ of compass point directions that appear in the $i$th row of $C$.
It can then be seen that
\begin{equation}\label{eqn-wenese}
\begin{array}{rcl}
    \#\WE_i &=& \#\NS_i+L_{i+1}-L_{i} \,;\cr
		\#\NE_i &=& L_i - \#\NS_i -\#\NW_i \,; \cr
		\#\SE_i &=& m-L_{i+1} -\#\NS_i - \#\SW_i \,.\cr
\end{array}
\end{equation}
The first of these identities follows from the fact that consecutive entries $1$ in row $i$ of the cumulative row 
sum matrix of $A$ give rise to a consecutive sequence of compass point entries $\WE$, $\SW$ or $\NW$ in row $i$ of $C$. 
Any entry to the immediate left of such a sequence is automatically $\NS$. However it will not be present if the 
leftmost entry in the first column is any one of $\WE$, $\SW$ or $\NW$. In such a case $L_{i+1}-L_i$ will be $1$. Hence the result follows.
The second identity follows from the fact that the column sum of entries of $A$ above the position of any
$\NS$, $\NE$ and $\NW$ in row $i$ of $C$ is necessarily $1$, and those of entries above the positions of $\WE$, $\SE$ and $\SW$ is $0$. 
The sum of all these column sums is therefore $\#\NS_i+\#\NE_i+\#\NW_i$, and this must coincide with the
sum of all row sums of $A$ for all rows above the $i$th. But this is just $L_i$ since a row sum of entries in 
$A$ is $1$ if the leftmost entry in $C$ is $\WE$, $\SW$ or $\NW$, and $0$ otherwise. 
The third identity is then just a consequence of the total number of entries in row $i$ being $m$.

In addition it should be noted that thanks to the half-turn symmetry of a $(2n+1)\times m$ $\lambda$-HTSASM we have 
\begin{equation}\label{eqn-Lsymmetry}
   L_{2n+2-i}=n-i+L_{i+1}
\end{equation}
for all $i=1,2,\ldots,2n+1$.

Using these identities the specification of weights in Theorem~\ref{the-result1} can be
re-expressed in the form
\begin{equation}
   \wgt(A) =  \prod_{i=1}^n (s_ix_i)^{n-i}  (s_ix_it_i\ov{y}_i)^{L_i} (z_0s_ix_i)^{L_{i+1}-L_i}\    \prod_{i=1}^{2n+1}\prod_{j=1}^m  \wgt(c_{ij})\,,
\end{equation}
with $\wgt(c_{ij})$ as tabulated below: 
\begin{equation}\label{tab-sxytz}
\begin{array}{|l|l|l|l|}
\hline
\hbox{Entry}&i<n+1&i=n&i>n+1\cr
\hbox{at $(i,j)$}&j\geq1&j\geq1&j\geq1\cr
\hline
\WE&1&1&1\cr
\NS&s_ix_i+\ov{t}_iy_i&z_0+\ov{z}_0&\ov{s}_{2n+2-i}\ov{x}_{2n+2-i}+t_{2n+2-i}\ov{y}_{2n+2-i}\cr
\NE&1&1&1\cr
\SE&1&1&1\cr
\NW&\ov{t}_i\,y_i&\ov{z}_0&\ov{s}_{2n+2-i}\ov{x}_{2n+2-i}\cr
\SW&s_i\,x_i&z_0&t_{2n+2-i}\ov{y}_{2n+2-i}\cr
\hline
\end{array}
\end{equation}
Here we have adopted a notation that will henceforth be used throughout the paper; 
namely the notation $\ov{z}=z^{-1}$ for any $z$.

\section{Primed shifted tableaux and lattice paths}
\label{sec:PSTandLPs}
Alternating sign matrices, compass point matrices and square ice configurations are three 
combinatorial objects that describe the same thing. There are at least two more types 
of combinatorial objects in this family: shifted tableaux~\cite{Macdonald,HKWeyl} 
and monotone triangles~\cite{Okadapartial} that are themselves equivalent to strict Gelfand-Tsetlin patterns~\cite{Tokuyama}. 
Of these we will focus our attention on what we call
unprimed shifted tableaux as they are the precursor to what we really need, 
namely primed shifted tableaux~\cite{Macdonald,HKWeyl}, and through them sets of non-intersecting lattice paths.

\begin{Definition}\label{def-T}
For $n\in\N$ let $\lambda=(\lambda_1, \lambda_2, \ldots, \lambda_n)$ be a partition with $n$ distinct non-zero parts. 
Then the set ${\cal T}^\lambda(2n+1)$ of all unprimed shifted tableaux $T$ of shape $\lambda$ over the alphabet
$$
1 < 2 < \cdots < n < 0 < \ov{n} < \cdots < \ov2 < \ov1
$$ 
consists of those $T$ which are an array of $n$ rows of boxes of lengths $\lambda_i$ for 
$i=1,2,\ldots,n$ left adjusted to a diagonal line
in which each box contains an entry from the above alphabet subject to the rules: 
\begin{itemize}
 \item entries weakly increase across rows from left to right and down columns from top to bottom;
 \item entries strictly increase down diagonals from top-left to bottom right; 
 \item exactly one of $\{k, \ov{k}\}$ appears on the main diagonal for all $k=1,2,\ldots,n$.
\end{itemize}
\end{Definition}

It might be noted that these rules imply that $0$ cannot appear as an entry on the main diagonal.
Such $T\in{\cal T}^\lambda(2n+1)$ are in bijective correspondence with all $\lambda$-HTSASMs $A\in{\cal B'}^\lambda_n$.
Quite generally the passage from an alternating sign matrix $A$ to a corresponding unprimed shifted tableau $T$
with entries from an alphabet $(e_1<e_2<\cdots\ )$ may be accomplished as follows. 
For each $A$ draw up a matrix $B$ whose entries are the right to left accumulated
row sums of the elements of $A$. That is to say $b_{ij}=\sum_{k\geq j} a_{ik}$. Then for each $j$ the entries in the 
$j$th diagonal of the shifted tableaux $T$ are found by reading down the $j$th column of $B$ from top to bottom.
If the $r$th non-zero entry $1$ appears in row $k=i_r$ of $B$ then one places $e_k$ in the $r$th position down 
the $j$th diagonal of $T$. 

For example, in the case $n=3$, $\lambda=(8,6,3)$, and alphabet 
$1<2<3<0<\ov3<\ov2<\ov1$ this correspondence between $A$ and $T$ is illustrated in Figure~\ref{fig-ABT-lambda}. 
It may be seen that the non-zero entries $1$ in the matrix $B$ of Figure~\ref{fig-ABT-lambda} 
correspond to the entries $\WE$, $\NW$ and $\SW$ in the compass point matrix of Figure~\ref{fig-ACS-lambda}
that is associated with the same $\lambda$-HTSASM $A$. It is these entries in $B$ that determine in a one-to-one way the 
entries of $T$.

\begin{figure*}
\begin{center}
\begin{equation*}
\begin{array}{c}
\begin{array}{c}
     1\\
     2\\
     3\\
     0\\
     \ov3\\
     \ov2\\
     \ov1\\
\end{array}
		\ A =
\left[
     \begin{array}{cccccccc}
     0 & 1 & 0 & 0 & 0 & 0 & 0 & 0\\
     0 &\ov1& 0 & 1 & 0 & 0 & 0 & 0\\
     1 & 0 & 0 & 0 & 0 & 0 & 0 & 0\\
   \ov1& 0 & 1 & \ov1 & 0 & 0 & 1 & 0\\
     0 & 0 & \ov1 & 0 & 1 & 0 & \ov1 & 1\\
     0 & 1 & 0 & 0 & \ov1 & 1 & 0 & 0\\
     0 & \ov1 & 1 & 0 & 0 & 0 & 0 & 0\\
     \end{array}\right]
\ B =
\left[
     \begin{array}{cccccccc}
     1 & 1 & 0 & 0 & 0 & 0 & 0 & 0\\
     0 & 0 & 1 & 1 & 0 & 0 & 0 & 0\\
     1 & 0 & 0 & 0 & 0 & 0 & 0 & 0\\
     0 & 1 & 1 & 0 & 1 & 1 & 1 & 0\\
     0 & 0 & 0 & 1 & 1 & 0 & 0 & 1\\
     1 & 1 & 0 & 0 & 0 & 1 & 0 & 0\\
     0 & 0 & 1 & 0 & 0 & 0 & 0 & 0\\
     \end{array}\right]
\cr\cr		
T\  = \
{\vcenter
  {\offinterlineskip
 \halign{&\mystrut\vrule#&\mybox{\hss$#$\hss}\cr
  \hr{17}\cr
                      &1&&1 &&2   && 2 && 0 && 0 && 0 &&\ov{3}  &\cr     
  \hr{17}\cr
                      \omit& &&3&&  0 && 0 && \ov3 && \ov3 && \ov2  &\cr          
  \nr{2}&\hr{13}\cr
                      \omit& &\omit& &&\ov2   &&\ov2 &&\ov1 &\cr      
  \nr{4}&\hr{7}\cr    
 }}}
\end{array}
\end{equation*}
\end{center}
\caption{Example of the transition from a $\lambda$-HTSASM $A$ to its associated unprimed shifted tableau $T$
by way of a cumulative row sum matrix $B$ in the case $N=2n+1$ with $n=3$ and $\lambda=(8,6,3)$.}
\label{fig-ABT-lambda}
\end{figure*}

As we have indicated these unprimed shifted tableaux are just a precursor to what we need, that is to say,
the set ${\cal P}^\lambda(2n+1)$ of primed shifted tableaux $P$. These are defined by:

\begin{Definition}\label{def-P}
For $n\in\N$ let $\lambda=(\lambda_1, \lambda_2, \ldots, \lambda_n)$ be a partition with $n$ distinct non-zero parts. 
Then the set ${\cal P}^\lambda(2n+1)$ of all primed shifted tableaux $P$ of shape $\lambda$ over the alphabet
$$
1'< 1 < 2' < 2 < \cdots < n' < n < 0 < \ov{n}' < \ov{n} < \cdots < \ov2' < \ov2 < \ov1' < \ov1
$$ 
consists of those $P$ that may be obtained from all $T$ in ${\cal T}^\lambda(2n+1)$ over the unprimed subset of
this alphabet by adding primes to entries in such a way that
\begin{itemize}
 \item entries weakly increase across rows from left to right and down columns from top to bottom;
 \item entries strictly increase down diagonals from top-left to bottom right; 
 \item no two identical primed entries appear in the same row;
 \item no two identical unprimed entries appear in the same column;
 \item no primes appear on the main diagonal.
\end{itemize}
\end{Definition}

Moreover, as we made explicit in \cite{HKbij} pp.\ 450-2, (albeit with a different convention for the labelling of compass point matrix elements) there is an association between entries $\SW$, $\NW$ and $\NS$ in the compass point matrix $C$ and primed and unprimed entries in the primed shifted tableau $P$.  Specifically, for an alphabet $(e_1<e_2<\cdots\ )$ of entries in a shifted tableau $T$, an entry $\SW$ in the $i$th row of $C$ is associated with an entry $e_i$ in $P$ that must remain unprimed, while an entry $\NW$ in the $i$th row of $C$ is associated with an entry $e'_i$ that must be primed. On the other hand, an entry $\NS$ in the $i$th row of $C$ may be associated with an entry $e_i$ or $e'_i$ in $P$. It corresponds to the first entry in a strip of $e_i$s in $T$ not starting on the main diagonal, and this entry has the freedom to be primed or unprimed in $P$.  Expressed more precisely, the entries in column $j$ of $C$ are associated with the entries in diagonal $j+1$ of $P$. This offset is related to the fact that each entry $\NS$ in row $i$ of $C$ is itself associated with an entry $\WE$ in the same row, and separated from it by a sequence of $\SW$ and $\NW$ entries. We have chosen to associate entries $e_i$ and $e'_i$ with $\NS$ rather than $\WE$ for reasons that will become clear when weights are assigned to these entries. This leaves the unprimed entries $e_i$ on the main diagonal of $P$ to be obtained from $C$. This time the association is such that each entry $\WE$, $\SW$ or $\NW$ in row $i$ and column $1$ of $C$ is associated with an unprimed entry $e_i$ on the main diagonal of $P$.

The final combinatorial objects required here are sets of non-intersecting lattice paths. These are constructed in such a 
way that they are in bijective correspondence with primed shifted tableaux. The lattice path grid is the set of 
points $(j,i)$ with $i=1,2,\ldots,n,0,\ov{n},\ldots,\ov2,\ov1,\ov0$ and $j=1,2,\ldots,m$, 
augmented by a line of points $(-n+k,0)$ with $k=1,2,\ldots,n$. For given strict partition $\lambda$ of length $\ell(\lambda)=n$
there are $n$ lattice paths extending from $(-n+k,0)$, via either $(k,1)$ or
$(\ov{k},1)$ by way of a curved edge, and thereafter to $(\lambda_i,\ov0)$ for some $i$ by way of horizontal left to right, vertical top to bottom or 
diagonal top left to bottom right edges for $k=1,2,\ldots,n$, as illustrated in Figure~\ref{fig-PtoLP1}. 
The map from each primed shifted tableaux $P$ with entries $e_1,e_2,\ldots,e_{\lambda_i}$ in row $i$
is such that the $i$th lattice path is comprised of a curved edge terminating at $(e_1,1)$, 
and either a horizontal or a downward diagonal edge terminating at $(e_j,j)$ according as $e_j$ is unprimed or primed, respectively,
for $j=2,\ldots,\lambda_i$, together with precisely those vertically downward edges that make the path continuous 
and such that it terminates at $(\lambda_i,\ov0)$.

\begin{figure*}[htbp]
\begin{center}
\begin{equation*}
\begin{array}{cc}
{\vcenter
  {\offinterlineskip
 \halign{&\mystrut\vrule#&\mybox{\hss$#$\hss}\cr
  \hr{17}\cr
                      &1&&1 &&2'   && 2 && 0 && 0 && 0 &&\ov{3}'  &\cr     
  \hr{17}\cr
                      \omit& &&3&&  0' && 0 && \ov3 && \ov3 && \ov2'  &\cr          
  \nr{2}&\hr{13}\cr
                      \omit& &\omit& &&\ov2   &&\ov2 &&\ov1' &\cr      
  \nr{4}&\hr{3}\cr                                                                                                                                         \omit& 
  \nr{6}&\hr{3}\cr  
 }}}
& \qquad 
\vcenter{\hbox{
\begin{tikzpicture}[x={(0in,-0.2in)},y={(0.2in,0in)}] 
\foreach \i in {1,...,7}
\foreach \j in {1,...,8}\normalsize
\draw(\i,\j)node{$\bullet$};
\foreach \j in {-2,...,0}
\draw(4,\j)node{$\bullet$};
\foreach \j in {1,...,8}
\draw(8,\j)node{$\bullet$};
\draw(1,-3)node{$1$};
\draw(2,-3)node{$2$};
\draw(3,-3)node{$3$};
\draw(4,-3)node{$0$};
\draw(5,-3)node{$\ov3$};
\draw(6,-3)node{$\ov2$};
\draw(7,-3)node{$\ov1$};
\draw(8,-3)node{$\ov0$};
\draw(9,1)node{$1$};
\draw(9,2)node{$2$};
\draw(9,3)node{$3$};
\draw(9,4)node{$4$};
\draw(9,5)node{$5$};
\draw(9,6)node{$6$};
\draw(9,7)node{$7$};
\draw(9,8)node{$8$};

\draw[draw=black,thick,out=90,in=180] (4,-2)to(1,1);
\draw[draw=black,thick] (1,1)to(1,2);
\draw[draw=black,thick]  (1,2)to(2,3);
\draw[draw=black,thick] (2,3)to(2,4);
\draw[draw=black,thick] (2,4)to(4,4);
\draw[draw=black,thick] (4,4)to(4,7);
\draw[draw=black,thick]  (4,7)to(5,8);
\draw[draw=black,thick] (5,8)to(8,8);

\draw[draw=black,thick,out=90,in=180] (4,0)to(3,1);
\draw[draw=black,thick] (3,1)to(4,2);
\draw[draw=black,thick] (4,2)to(4,3);
\draw[draw=black,thick]  (4,3)to(5,3);
\draw[draw=black,thick] (5,3)to(5,5);
\draw[draw=black,thick]  (5,5)to(6,6);  
\draw[draw=black,thick] (6,6)to(8,6);

\draw[draw=black,thick,out=270,in=180] (4,-1)to(6,1);
\draw[draw=black,thick] (6,1)to(6,2);
\draw[draw=black,thick]  (6,2)to(7,3);
\draw[draw=black,thick](7,3)to(8,3); 
\end{tikzpicture}
}}
\end{array}
\end{equation*}
\end{center}
\caption{Example of  the lattice paths for a given primed shifted tableau.}
\label{fig-PtoLP1}
\end{figure*}
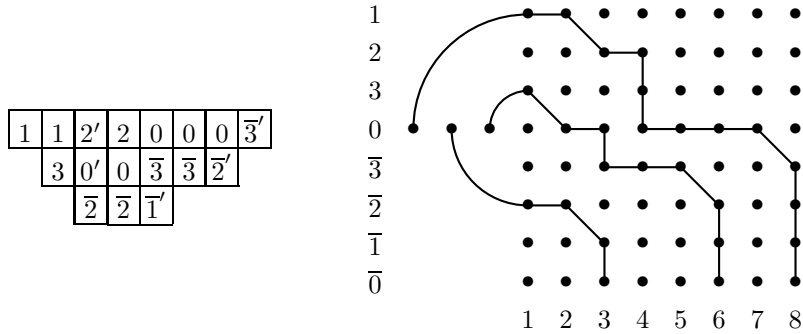

One can see in this example that the leftmost $\ov3$ in the primed shifted tableau could
equally well be primed. If it is primed then the lattice paths remain non-intersecting
as shown in Figure~\ref{fig-PtoLP2}.
On the other hand if the leftmost $\ov3$ is replaced by a $0$ the resulting primed shifted tableau is non-standard
and the corresponding lattice paths intersect as shown in Figure~\ref{fig-PtoLP3}.

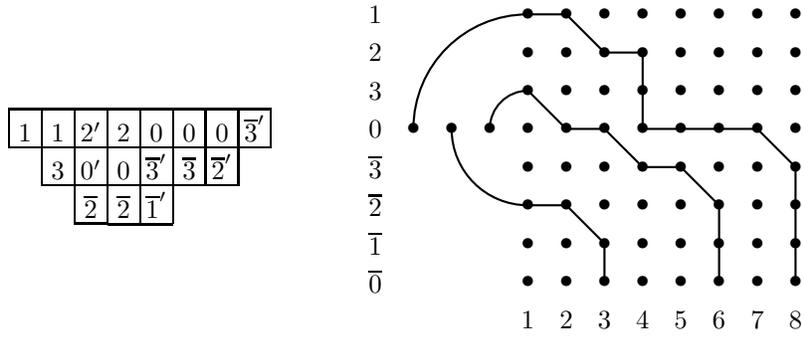
\begin{figure*}[htbp]
\begin{center}
\begin{equation*}
\begin{array}{cc}
{\vcenter
  {\offinterlineskip
 \halign{&\mystrut\vrule#&\mybox{\hss$#$\hss}\cr
  \hr{17}\cr
                      &1&&1 &&2'   && 2 && 0 && 0 && 0 &&\ov{3}'  &\cr     
  \hr{17}\cr
                      \omit& &&3&&  0' && 0 && \ov3' && \ov3 && \ov2'  &\cr          
  \nr{2}&\hr{13}\cr
                      \omit& &\omit& &&\ov2   &&\ov2 &&\ov1' &\cr      
  \nr{4}&\hr{3}\cr                                                                                                                                         \omit& 
  \nr{6}&\hr{3}\cr  
 }}}
& \qquad 
\vcenter{\hbox{
\begin{tikzpicture}[x={(0in,-0.2in)},y={(0.2in,0in)}] 
\foreach \i in {1,...,7}
\foreach \j in {1,...,8}\normalsize
\draw(\i,\j)node{$\bullet$};
\foreach \j in {-2,...,0}
\draw(4,\j)node{$\bullet$};
\foreach \j in {1,...,8}
\draw(8,\j)node{$\bullet$};
\draw(1,-3)node{$1$};
\draw(2,-3)node{$2$};
\draw(3,-3)node{$3$};
\draw(4,-3)node{$0$};
\draw(5,-3)node{$\ov3$};
\draw(6,-3)node{$\ov2$};
\draw(7,-3)node{$\ov1$};
\draw(8,-3)node{$\ov0$};
\draw(9,1)node{$1$};
\draw(9,2)node{$2$};
\draw(9,3)node{$3$};
\draw(9,4)node{$4$};
\draw(9,5)node{$5$};
\draw(9,6)node{$6$};
\draw(9,7)node{$7$};
\draw(9,8)node{$8$};

\draw[draw=black,thick,out=90,in=180] (4,-2)to(1,1);
\draw[draw=black,thick] (1,1)to(1,2);
\draw[draw=black,thick]  (1,2)to(2,3);
\draw[draw=black,thick] (2,3)to(2,4);
\draw[draw=black,thick] (2,4)to(4,4);
\draw[draw=black,thick] (4,4)to(4,7);
\draw[draw=black,thick]  (4,7)to(5,8);
\draw[draw=black,thick] (5,8)to(8,8);

\draw[draw=black,thick,out=90,in=180] (4,0)to(3,1);
\draw[draw=black,thick] (3,1)to(4,2);
\draw[draw=black,thick] (4,2)to(4,3);
\draw[draw=black,thick]  (4,3)to(5,4);
\draw[draw=black,thick] (5,4)to(5,5);
\draw[draw=black,thick]  (5,5)to(6,6);  
\draw[draw=black,thick] (6,6)to(8,6);

\draw[draw=black,thick,out=270,in=180] (4,-1)to(6,1);
\draw[draw=black,thick] (6,1)to(6,2);
\draw[draw=black,thick]  (6,2)to(7,3);
\draw[draw=black,thick](7,3)to(8,3); 
\end{tikzpicture}
}}
\end{array}
\end{equation*}
\end{center}
\caption{Example of the lattice paths for a second shifted primed tableau.}
\label{fig-PtoLP2}
\end{figure*}

\begin{figure*}[htbp]
\begin{center}
\begin{equation*}
\begin{array}{cc}
\begin{array}{c}
{\vcenter
  {\offinterlineskip
 \halign{&\mystrut\vrule#&\mybox{\hss$#$\hss}\cr
  \hr{17}\cr
                      &1&&1 &&2'   && 2 && 0 && 0 && 0 &&\ov{3}'  &\cr     
  \hr{17}\cr
                      \omit& &&3&&  0' && 0 && 0 && \ov3 && \ov2'  &\cr          
  \nr{2}&\hr{13}\cr
                      \omit& &\omit& &&\ov2   &&\ov2 &&\ov1' &\cr      
  \nr{4}&\hr{7}\cr  
 }}}
\cr\cr
{\vcenter
  {\offinterlineskip
 \halign{&\mystrut\vrule#&\mybox{\hss$#$\hss}\cr
  \hr{13}\cr
                      &1&&1 &&2'   && 2 && \ov3 && \ov2'  &\cr     
  \hr{19}\cr
                      \omit& &&3&&  0' && 0 && 0 && 0 && 0 && 0 && \ov3'  &\cr          
  \nr{2}&\hr{17}\cr
                      \omit& &\omit& &&\ov2   &&\ov2 &&\ov1' &\cr      
  \nr{4}&\hr{7}\cr  
 }}}
\end{array}
& \ 
\vcenter{\hbox{
\begin{tikzpicture}[x={(0in,-0.2in)},y={(0.2in,0in)}] 
\foreach \i in {1,...,7}
\foreach \j in {1,...,8}\normalsize
\draw(\i,\j)node{$\bullet$};
\foreach \j in {-2,...,0}
\draw(4,\j)node{$\bullet$};
\foreach \j in {1,...,8}
\draw(8,\j)node{$\bullet$};
\draw(1,-3)node{$1$};
\draw(2,-3)node{$2$};
\draw(3,-3)node{$3$};
\draw(4,-3)node{$0$};
\draw(5,-3)node{$\ov3$};
\draw(6,-3)node{$\ov2$};
\draw(7,-3)node{$\ov1$};
\draw(8,-3)node{$\ov0$};
\draw(9,1)node{$1$};
\draw(9,2)node{$2$};
\draw(9,3)node{$3$};
\draw(9,4)node{$4$};
\draw(9,5)node{$5$};
\draw(9,6)node{$6$};
\draw(9,7)node{$7$};
\draw(9,8)node{$8$};

\draw[draw=black,thick,out=90,in=180] (4,-2)to(1,1);
\draw[draw=black,thick] (1,1)to(1,2);
\draw[draw=black,thick]  (1,2)to(2,3);
\draw[draw=black,thick] (2,3)to(2,4);
\draw[draw=black,thick] (2,4)to(4,4);
\draw[draw=black,thick] (4,4)to(4,7);
\draw[draw=black,thick]  (4,7)to(5,8);
\draw[draw=black,thick] (5,8)to(8,8);

\draw[draw=black,thick,out=90,in=180] (4,0)to(3,1);
\draw[draw=black,thick] (3,1)to(4,2);
\draw[draw=black,thick] (4,2)to(4,3);
\draw[draw=black,thick]  (4,3)to(4,4);
\draw[draw=black,thick] (4,4)to(5,4);
\draw[draw=black,thick] (5,4)to(5,5);
\draw[draw=black,thick]  (5,5)to(6,6);  
\draw[draw=black,thick] (6,6)to(8,6);

\draw[draw=black,thick,out=270,in=180] (4,-1)to(6,1);
\draw[draw=black,thick] (6,1)to(6,2);
\draw[draw=black,thick]  (6,2)to(7,3);
\draw[draw=black,thick](7,3)to(8,3); 
\end{tikzpicture}
}}
\end{array}
\end{equation*}
\end{center}
\caption{Example of intersecting lattice paths corresponding to a pair of non-standard shifted primed tableaux.}
\label{fig-PtoLP3}
\end{figure*}
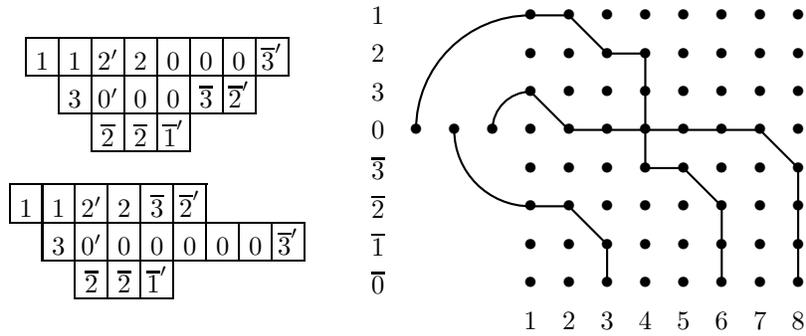

It can be seen that in this particular case there are two distinct primed shifted tableaux that give rise to the same set of lattice path edges depending upon
how the edges are assigned to each path beyond the point of intersection. However both primed shifted tableaux are non-standard. Moreover the
two choices of lattice paths correspond to a transposition of the end points. It is a characteristic of all such intersecting lattice paths that they 
can be associated in pairs through the identification of their rightmost point of intersection, and that the pairs differ in a simple transposition
of their end points.

\section{Weights of primed shifted tableaux and sets of non-intersecting lattice paths}
\label{sec:PSTandLPwgts}

The next step towards our result is the adoption of a sufficiently general weighting of the compass points matrix. 
Brubaker and Schultz \cite{BrubakerSchultz} have defined a universal weighting for the compass points matrix for a large 
number of different types of HTSASMs (i.e.\ not just for the odd-sided ones we consider here). We reinterpret this universal weighting 
in a form more amenable to its subsequent use here in connection with primed shifted tableaux as follows.
\begin{Definition}\label{def-Awgt}
For $n\in\N$ let $\x=(x_1,x_2,\ldots,x_{n})$ and $\y=(y_1,y_2,\ldots,y_{n})$ be two sequences of non-zero parameters, 
and let  $z_0$ be a further non-zero parameter. Let $\lambda$ be a strict partition
of length $\ell(\lambda)=n$ and first part $\lambda_1=m \geq n$. Then for any $(2n+1)\times m$
$\lambda$-HTSASM $A\in{\cal B'}^\lambda_n$ let the corresponding compass point matrix be $C$.
Then we define
\begin{equation}\label{eqn-Awgt}
   \wgt(A) = \prod_{i=1}^n x_i^{n-i}\ (x_i\ov{y}_i)^{L_i} \ (z_0x_i)^{L_{i+1}-L_i}\  \prod_{i=1}^{2n+1}\prod_{j=1}^m \wgt(c_{ij})
\end{equation}
where $\wgt(c_{ij})$ is as tabulated below
\begin{equation}\label{tab-Cwgt}
\begin{array}{|l|l|l|l|}
\hline
\hbox{Entry}&i<n+1&i=n+1&i>n+1\cr
\hbox{at $(i,j)$}&j\geq1&j\geq1&j\geq1\cr
\hline
\WE&1&1&1\cr
\NS&x_i+y_i&z_0+\ov{z}_0&\ov{y}_{2n+2-i}+\ov{x}_{2n+2-i}\cr
\NE&1&1&1\cr
\SE&1&1&1\cr
\NW&y_i&\ov{z}_0&\ov{x}_{2n+2-i}\cr
\SW&x_i&z_0&\ov{y}_{2n+2-i}\cr
\hline
\end{array}
\end{equation}
\end{Definition}

This may appear somewhat less general than that used in the statement of Theorem~\ref{the-result1},
but that is not the case since the tabulation (\ref{tab-sxytz}) may be recovered from that of (\ref{tab-Cwgt})
by mapping $x_i$ to $s_ix_i$ and $y_i$ to $\ov{t}_iy_i$ for $i=1,2,\ldots,n$.
 
The great merit of adding primes to the unprimed shifted tableaux is that for
each $T\in{\cal T}^\lambda(2n+1)$ corresponding to a $\lambda$-HTSASM $A$ there exist precisely $2^{\neg(A)}$ distinct 
primed shifted tableaux $P\in{\cal P}^\lambda(2n+1)$ where $\neg(A)$ is the number of entries $-1$ in $A$,
or equivalently. the number of entries $\NS$ in the corresponding compass point matrix $C$. This means that 
the $2^{\neg(A)}$ terms that arise in the expression for $\wgt(A)$ based on the weights $\wgt(c_{ij})$ 
tabulated in (\ref{tab-Cwgt}) can be separated into a sum of monomials with just one for each distinct shifted primed tableau
$P$. To be precise ,the above weighting in equation (\ref{tab-Cwgt}) of compass point matrices $C$ leads 
directly to the following:

\begin{Lemma}\label{lem-AP}
With the notation of Definition~\ref{def-Awgt}
\begin{equation}\label{eqn-AP}
\sum_{A\in{\cal B'}^\lambda_n} \ \wgt(A)  = \prod_{i=1}^n  x_i^{n-i}\  \sum_{P\in{\cal P}^\lambda(2n+1)} \ \wgt(P)
\end{equation}
where
\begin{equation}
   \wgt(P) =  \prod_{e\in P}\ \wgt(e)\,,
\end{equation}
where for each entry $e$ in $P$ the map from $e$ to its weight $\wgt(e)$ 
is given by: 
\begin{equation}\label{tab-Pwgt}
\begin{array}{|l|l|}
\hline
\hbox{Diagonal weights}&\hbox{Off-diagonal weights} \cr
\hline
 k\mapsto z_0x_k \prod_{j>k} (x_j\ov{y}_j) &k\mapsto x_k\cr
 \ov{k}\mapsto 1    &k'\mapsto y_k\cr
                    &\ov{k}\mapsto \ov{y}_k\cr
                    &\ov{k}'\mapsto \ov{x}_k\cr
                    &0\mapsto z_0\cr
                    &0'\mapsto \ov{z}_0\cr
\hline
\end{array}
\end{equation}
\end{Lemma}

\begin{proof}
It suffices to note that under the map from $A$ to $P$, by way of the accumulated row sum matrix $B$, the 
compass point matrix $C$ and the shifted tableau $T$, 
entries $\SW$ and $\NW$ in row $i$ of $C$ give rise to entries $k$ and $k'$,
respectively, in $P$ if $i=k<n+1$, and to $0$ and $0'$ if $i=n+1$, and to $\ov{k}$ and $\ov{k}'$ if $i=2n+2-k>n+1$. 
As can be seen from table~(\ref{tab-Cwgt}) each entry $\NS$ in any row $i$ of $C$ has weight equal to the 
sum of the weights of entries $\SW$ and $\NW$ in the same row. These weights are then ascribed to a
corresponding entry in $T$ that, as we have seen, may be unprimed or primed under the map from $A$ to $P$. Such an entry
is the first of the continuous strip of identical entries in $T$ arising from the sequence of $1$s in $B$ 
that extends as far as the position of the next entry $\WE$ in $C$. However, this does not apply
to strips starting on the main diagonal. These arise if the leftmost entry in row $i$ of $C$ is either 
$\WE$, $\SW$ or $\NW$, so that there is no room for any entry $\NS$ to their left. The condition for this to 
happen is $L_{i+1}-L_i=1$. The contribution to the weights of the entries on the main diagonal of $P$ is then 
determined by the product $\prod_{i=1}^n (x_i\ov{y}_i)^{L_i} (z_0x_i)^{L_{i+1}-L_i}$ appearing in (\ref{eqn-Awgt}).
These factors could be distributed more or less arbitrarily between the weights of $k$ and $\ov{k}$ but we have chosen 
to set the weight of $\ov{k}$ to be $1$, thereby fixing the weight of $k$ as shown in (\ref{tab-Pwgt}).
\end{proof}

\section{Lattice path result}
\label{sec:LPresult}

Before proving our main result in Section \ref{sec:Tok} we need an intermediate result that we prove using lattice path techniques.

\begin{Theorem}\label{the-Pdet}
For all $n\in\N$ let $\lambda$  
be a strict partition of length $\ell(\lambda)=n$.  
Then with our weighting of primed shifted tableaux given in (\ref{tab-Pwgt}) we have
\begin{equation}\label{eqn-Pdet}
   \sum_{P\in{\cal P}^\lambda(2n+1)} \wgt(P) = [q^\lambda] \det_{1\leq k,\ell\leq n} \big(\,h_{k,\ell}(q_\ell)\,\big)  
\end{equation}
where
\begin{equation}\label{eqn-hfg}
            h_{k,\ell}(q_\ell)=f_{k,\ell}(q_\ell) + (-1)^{n-k}g_{k,\ell}(q_\ell)
\end{equation}
with
\begin{equation}
 f_{k,\ell}(q_\ell)= \frac{\u_kq_\ell}{1-x_kq_\ell}\ \prod_{i=k+1}^n \frac{1+y_iq_\ell}{1-x_iq_\ell}\ \ \frac{1+\ov{z}_0q_\ell}{1-z_0q_\ell}\  \ 
 \prod_{i=1}^n \frac{1+\ov{x}_iq_\ell}{1-\ov{y}_iq_\ell}.
 \label{eqn-f}
\end{equation}
and
\begin{equation}
g_{k,\ell}(q_\ell)=\frac{\ov{v}_kq_\ell}{1-\ov{y}_kq_\ell}\ \prod_{i=1}^{k-1} \frac{1+\ov{x}_iq_\ell}{1-\ov{y}_iq_\ell}\,. 
\label{eqn-g}
\end{equation}
where 
\begin{equation}\label{eqn-uv}
  u_k= z_0x_k\prod_{i=k+1}^n x_i\ov{y}_i \quad\hbox{and}\quad  \ov{v}_k=1\,.
\end{equation}
\end{Theorem}

\begin{proof}
To prove this result we construct a lattice path argument along the lines of that employed by Okada\footnote{Note that we could have also followed a different path, so to speak, making use of Theorem 4.2, equation (32), of Ishikawa and Wakayama \cite{IWIII} instead.  Their theorem is a more general case of Stembridge's result \cite{Stembridge}; however, their theorem holds in the case when the starting and ending points are not $D$-compatible (as defined in \cite{Stembridge}), which is the situation we would construct here.  The connection to \cite{IWIII} and \cite{Stembridge} implies there would be Pfaffian results as well, and these may be the subject of some of our future work.} \cite{Okadapartial}.
The lattice paths have starting points $P_i=(-n+i,0)$ and ending points $Q_j=(\lambda_j,\ov0)$ for $i,j=1,2,\ldots,n$, 
and for each primed shifted tableau $P$ extend from $P_i$ to $Q_{\pi(i)}$ for some permutation $\pi$ of $(1,2,\ldots,n)$.
With each vertical edge weighted $1$, the weights on curved, horizontal, and diagonal edges are as follows: 

1) For a curved edge from $(-n+k, 0)$ to $(0,k)$ the weight is $u_k$, 
and from $(-n+k,0)$ to $(0,\ov{k})$, the weight is $\ov{v}_k$. 

2) For a horizontal edge for unbarred nonzero $i$ from $(j,i)$ to $(j+1,i)$, the weight is $x_i$.  
For a horizontal edge for barred nonzero $i$ from $(j,\bar{\imath})$ to $(j+1,\bar{\imath})$, the weight is $\ov{y}_i$.  
For a horizontal edge for $0$ from $(j,0)$ to $(j+1,0)$, the weight is $z_0$.  

3) For a diagonal edge for unbarred nonzero $i$ from $(j,i-1)$ to $(j+1, i)$ the weight is $y_i$.  
For a diagonal edge for barred nonzero $i$ from $(j,\ov{i+1})$ to $(j, \bar{\imath})$ the weight is $\ov{x}_i$. 
For a diagonal edge for $0$ from  $(j,n)$ to $(j+1,0)$ the weight is $\ov{z}_0$.

It is not hard to see that with this weighting, the set of non-intersecting lattice paths from $P_i$ to $Q_i$ gives 
the generating function for the required sum of weights of primed shifted tableaux on the left hand side of (\ref{eqn-Pdet}).
In particular, the row starting with $k$  or $\ov{k}$ corresponds to the lattice path starting at $(-n+k,0)$. 

It remains to show that the right hand side of (\ref{eqn-Pdet}) counts all possible lattice paths (intersecting and non-intersecting) of this form.  To do this we show that $f$ and $g$ as defined in equations (\ref{eqn-f}) and (\ref{eqn-g}) are the generating functions for  the weights of individual paths starting at $k$ and $\ov{k}$ respectively.  Then from Okada's argument in \cite{Okadapartial} about cancellations between pairs of intersecting paths we know that the determinant of these functions counts the weights of the set of non-intersecting paths.  The final step is to combine determinants.

First we show that the generating function for the weights of all possible rows starting with $k$ in the tableaux can be expressed as
\begin{equation}
 f_{k,\ell}(q_\ell):=\frac{u_kq_\ell}{1-x_kq_\ell}\ 
\prod_{i=k+1}^n \frac{1+y_iq_\ell}{1-x_iq_\ell}\ \ \frac{1+\ov{z}_0q_\ell}{1-z_0q_\ell}\ \
 \prod_{i=1}^n \frac{1+\ov{x}_iq_\ell}{1-\ov{y}_iq_\ell}.
 \label{eqnf}
\end{equation}

As justification for this assertion, we point out that the first (curved) step is at height $k$ and has 
weight $u_k$.   After that we cannot have any diagonal steps at height $k$ but we can have more horizontal steps.  
Their weights are generated by $(1-x_kq_\ell)^{-1}$.  Then we can have the remaining steps from any height bigger than $k$.  
We can have at most one of these as a diagonal step---and these weights are generated by $(1+y_iq_\ell)$ for unbarred entries, 
by $(1+\ov{z}_0q_\ell)$ for zero entries, and by $(1+\ov{x}_iq_\ell)$ for barred entries---%
or as many horizontal steps---and these are counted by $(1-x_iq_\ell)^{-1}$ for unbarred entries, by $(1-z_0q_\ell)$ for zero entries,  
and by $(1-\ov{y}_iq_\ell)^{-1}$ for barred entries.  

A similar argument implies that 
the generating function for the weights of all possible rows 
starting with $\ov{k}$ in the tableaux can be expressed as
\begin{equation}
g_{k,\ell}(q_\ell):=\frac{\ov{v}_kq_\ell}{1-\ov{y}_kq_\ell}\ \prod_{i=1}^{k-1} \frac{1+\ov{x}_iq_\ell}{1-\ov{y}_iq_\ell}\,. 
\label{eqng}
\end{equation}

In the expansion of each of these generating functions $f_{k,\ell}(q_\ell)$ and $g_{k,\ell}(q_\ell)$ the parameter
$q_\ell$ carries the total number of steps to the right in each lattice path, and thereby its exponent gives the
the number of the column in which each lattice path terminates. 
In dealing with lattice paths associated with primed shifted tableaux of shape $\lambda$ it is convenient to use the notation 
whereby $[q_{\ell}^{\lambda_{\ell}}]$ picks out the coefficient of $q_{\ell}^{\lambda_{\ell}}$ in the expansion 
of these generating functions and $[q^\lambda]$ the coefficient of $q_{1}^{\lambda_{1}}q_2\,^{\lambda_2}\cdots
q_n^{\lambda_n}$.

To put equations (\ref{eqnf}) and (\ref{eqng}) together to handle the entirety of paths, we begin by using the notation of 
Okada to refer to the sequence of entries $\p=(d_1,d_2,\ldots,d_n)$ 
on the main diagonal of $P$ as its {\em profile}, and note that there are $2^n$ distinct profiles $\p$ of those 
$P$ of shape $\lambda$, where each such profile $\p$ contains either $k$ or $\ov{k}$ but not both. 
It follows from Okada's argument~\cite{Okadapartial} that
the generating function for all non-intersecting paths corresponding to a
particular profile $\p$ is given by $\ds \det(\,f_{k\ell}(q_\ell)\,)$ for $k\in\p$
and $\ds \det(\,g_{k\ell}(q_\ell)\,)$ for $\ov{k}\in\p$.
							
Let $\pi$ be the permutation mapping $\p$ to $\p^{\pi}$ where $\p^{\pi}_k=k$ or $\ov{k}$.
For each allowed profile $\p$ the permutation $\pi$ is a product of cycles having parity $(-1)^{n-k}$
moving $\ov{k}$ to position $k$. Hence summing over all distinct profiles $\p$ and permuting rows
and combining determinants one arrives at  
$
           \sum_{\p}  \det(\,\tilde{h}_{k\ell}(q_\ell)\,)  =  \det(\, f_{k\ell}(q_\ell)+(\!-1)^{n-k}g_{k\ell}(q_\ell) \,)
$. The proof is completed by restricting attention to those primed shifted tableaux of shape $\lambda$ in which case
one picks out the coefficient of $q_{\ell}^{\lambda_{\ell}}$ from each term in the $\ell$th column. Since $q_\ell$
only appears in this column, the operator $[q^{\lambda}]$ may be taken out of the determinant yielding the required result. \end{proof}

Now it remains to expand the determinant on the right hand side of (\ref{eqn-Pdet}).

\section{Determinantal expansion}
\label{sec:det}

The determinant appearing in Theorem~\ref{the-Pdet} does not, as it stands, appear to be one that
has already been evaluated in the extensive collection of various types of determinant in the papers 
of Krattenthaler~\cite{Kratdet1,Kratdet2} or indeed elsewhere. 
However, explicit evaluation in the cases $n=1$, $2$ and $3$ suggests the validity of the following
Lemma which we will state and prove:
\begin{Lemma}
For all $n\in\N$ let $\z=(x_1,\ldots,x_n,z_0,y_n,\ldots,y_2,y_1)$ and let $\q=(q_1,q_2,\ldots,q_n)$.
Then for $h_{k,\ell}(q_\ell)$ as given in Theorem~\ref{the-Pdet} we have
\begin{equation}
\det_{1\leq k,\ell\leq n}\big(\,h_{k,\ell}(q_\ell)\,\big) = Z(\z)\ K(\z,\q)\ Q(\q)
\label{eqn-deth}
\end{equation}
where
\begin{equation}
\begin{array}{rcl}
Z(\z) &=& \ds \prod_{i=1}^n x_i^{-(n-i)}\ \prod_{i=1}^n (1+z_0x_i) \  \prod_{1\leq i<j\leq n} (1+ x_ix_j)(1+x_i\ov{y}_j)\,; \cr\
K(\z,\q) &=& \ds 1\big/\big(\ \prod_{i=1}^n (1-q_iz_0)\ \prod_{i,j=1}^n (1-q_jx_i)(1-q_j\ov{y}_i) \ \big)\,;\cr
Q(\q)   &=& \ds \prod_{i=1}^n q_i\ \prod_{1\leq i<j\leq n} (q_i-q_j)  \ \prod_{1\leq i\leq j\leq n} (1+q_iq_j)\,. \cr
\end{array}
\label{eqn-ZKQ}
\end{equation}
\label{lem-deth}
\end{Lemma}

\begin{proof}
First it is convenient to extract a factor of
\begin{equation*}
  K(\z;q_\ell):=1\big/ \big( (1-q_\ell z_0)\ \prod_{i=1}^n (1-q_\ell x_i)(1-q_\ell\ov{y}_i) \big)
\end{equation*}
from each element in the $\ell$th column of the determinant, together with a factor of $q_\ell^{n+1}$,
and to extract a factor of $u_k$ from each element in the $k$th row. 
This gives
\begin{equation}\label{eqn-hKtf} 
 \det_{1\leq k,\ell\leq n}  \big(\,h_{k,\ell}(q_\ell)\,\big) 
 = K(\z,\q)\ \prod_{i=1}^n q_i^{n+1}u_i\ \det_{1\leq k,\ell\leq n}  \left(\,\tilde{f}_{k,\ell}(q_\ell)-\tilde{f}_{k,\ell}(-\ov{q}_\ell)\,\right) 
\end{equation}
with
\begin{equation}\label{eqn-tf}
\tilde{f}_{k,\ell}(q_\ell) = q_{\ell}^{-n} (1+\ov{z}_0q_\ell)\ \prod_{i=1}^{n} (1+\ov{x}_iq_\ell)\ 
\prod_{i=1}^{k-1}(1-x_iq_\ell)\  \prod_{i=k+1}^n (1+y_iq_\ell) \,.
\end{equation}
The two terms in the determinant owe their origin to identities: 
\begin{equation*}
\begin{array}{rcl}
f_{k,\ell}(q_\ell) &=&\ds  K(\z;q_\ell)\ u_kq_\ell^{n+1}\  \tilde{f}_{k,\ell}(q_\ell)  \,;\cr\cr
g_{k,\ell}(q_\ell) &=&\ds (-1)^{n-k+1}K(\z;q_\ell)\bigg(\ov{v}_k\,z_0x_k\!\prod_{i=k+1}^n\!x_i\ov{y}_i\bigg) q_\ell^{n+1}\ \tilde{f}_{k,\ell}(-\ov{q}_\ell)\,,
\end{array}
\end{equation*}
and the fact that 
\begin{equation}\label{eqn-uvratio}
   \ov{v}_k\,z_0x_k\!\prod_{i=k+1}^n\!x_i\ov{y}_i =u_k\,,
\end{equation}
as follows from the specification of $u_k$ and $v_k$ given in (\ref{eqn-uv}).

One advantage of the form (\ref{eqn-hKtf}) is that it shows immediately that the determinant vanishes if
$q_\ell=-\ov{q}_\ell$, since in such a case all elements in the $\ell$th column vanish, or if either $q_i=q_j$ 
or $q_i=-\ov{q}_j$ for any $i\neq j$, since in either case the $i$th and $j$th columns are proportional 
to one another. This is sufficient to show that the factors $(q_i-q_j)$
for $1\leq i<j\leq n$ and $(1+q_iq_j)$ for $1\leq i\leq j\leq n$, are necessarily factors of the given 
determinant. 

However, a more significant advantage of this form is that the determinant on the right-hand side
is strikingly similar to the one evaluated by Rosengren and Schlosser~\cite{RosengrenSchlosser}
in their Corollary~5.8, and contained in the compendium assembled by Krattenthaler~\cite{Kratdet2} as
his Lemma~19. Indeed, applying this Corollary to a particular sequence of 
polynomials, $P_{j-1}(x_i)=\prod_{k=1}^{j-1}(1+b_kx_i)$ of degree at most $j-1$ for 
$j=1,2,\ldots,n$, along with a change of parameters whereby $x_i$, $a_i$, $b_i$ and $c_i$ 
are mapped to $Iq_i$, $Ia_i$, $Ib_i$ and $Ic_i$ with $I=\sqrt{-1}$, yields the following:
\begin{Lemma}\label{lem-detm}
For $n\in\N$ let $\q=(q_1,q_2,\ldots,q_n)$, $\a=(a_1,a_2,\ldots,a_n)$,
$\b=(b_1,b_2,\ldots,b_n)$ and $\c=(c_1,c_2,\ldots,c_n,c_{n+1})$
be four sequences of indeterminates, with $q_i\neq0$ and $\ov{q}_i=q_i^{-1}$ for $i=1,2,\ldots,n$,
and set $m_{k,\ell}(q_\ell)=\tilde{g}_{k,\ell}(q_\ell)-\tilde{g}_{k,\ell}(-\ov{q}_\ell)$
where
\begin{equation}\label{eqn-tg}
\tilde{g}_{k,\ell}(q_\ell) = q_\ell^{-n}\ \prod_{i=1}^{n+1}\,(1+c_iq_\ell)\ \prod_{i=1}^{k-1}\,(1-b_iq_\ell)\,  \prod_{i=k+1}^n (1+a_iq_\ell) \,.
\end{equation}
Then
\begin{equation}\label{eqn-detm}
\begin{array}{l}
\ds \det_{1\leq k,\ell\leq n}  \big(\, m_{k,\ell}(q_\ell)\,\big) 
= \prod_{1\leq i<j\leq n+1}(1+c_ic_j)\  \prod_{1\leq i<j\leq n}(b_i+a_j) \ \cr\cr
\ds ~~~~\times~~ \prod_{i=1}^n\ q_i^{-n}(1+q_i^2)\ \prod_{1\leq i<j\leq n}(q_i-q_j)(1+q_iq_j)\,.
\end{array}
\end{equation}
\end{Lemma}

By comparing the expressions  for $\tilde{f}_{k,\ell}(q_\ell)$ and $\tilde{g}_{k,\ell}(q_\ell)$ in (\ref{eqn-tf}) 
and (\ref{eqn-tg}), respectively, 
it can be seen that setting $a_i=y_i$, $b_i=x_i$ and $c_i=\ov{x}_i$  
for $i=1,2,\ldots,n$ and $c_{n+1}=\ov{z}_0$ in Lemma~\ref{lem-detm}  gives
\begin{equation*}
\begin{array}{l}
\ds \det_{1\leq k,\ell\leq n}  \left(\,\tilde{f}_{k,\ell}(q_\ell) - \tilde{f}_{k,\ell}(-\ov{q}_\ell)\,\right) \cr\cr
\ds = Q(\q)\ \prod_{i=1}^n q_i^{-(n+1)}\,(1+\ov{z}_0\ov{x}_i) \prod_{1\leq i<j\leq n}\!\!\!(1+\ov{x}_i\ov{x}_j) \prod_{1\leq i<j\leq n}\!\!\!(x_i+y_j)\cr\cr
\ds = Q(\q)\ \prod_{i=1}^n q_i^{-(n+1)}\,u_i^{-1}\,\ov{x}_i^{n-i}\,(1+z_0x_i) \prod_{1\leq i<j\leq n}\!\!\!(1+x_ix_j)
          \prod_{1\leq i<j\leq n}\!\!\!(1+x_i\ov{y}_j)\cr\cr
\ds = Q(\q)\ Z(\z)\ \prod_{i=1} q_i^{-(n+1)} u_i^{-1}\,.
\end{array}
\end{equation*}
Using this in (\ref{eqn-hKtf}) immediately yields (\ref{eqn-deth}), thereby completing the proof of
Lemma~\ref{lem-deth}.
\end{proof}

An alternative more self-contained proof of Lemma~\ref{lem-detm}, avoiding any reliance on 
Corollary~5.8 of Rosengren and Schlosser~\cite{RosengrenSchlosser}, is provided here in Appendix~A.

\section{Schur functions and universal characters}
\label{sec:characters}

In order to recognise our various weighted sums of ASMs as being deformations of Weyl
character and denominator formulae, it is necessary to introduce some classical results on 
these topics. In particular, for the results in Section \ref{sec:Tok} we will need a number of 
well-known properties of the $GL(n)$ characters known as Schur functions~\cite{Littlewood,Macdonald},
and their extension to the case of $O(2n+1)$ characters.

For each $n\in\N$ and each partition $\mu$ of length $\ell(\mu)\leq n$ 
there exists 
an irreducible representation $V_{GL(n)}^\mu$ of $GL(n)$ of highest weight $\mu$, whose character evaluated for
a group element with eigenvalues $\x=(x_1,\ldots,x_n)$ is given by
\begin{equation}
  \ch V_{GL(n)}^\mu  = s_\mu(\x)\,,
\end{equation}
where, if we let $\delta=(n,n-1,\ldots,1)$, the Schur function $s_\mu(x)$ 
may be defined for an arbitrary sequence of indeterminates $\x=(x_1,x_2,\ldots,x_n)$ by
\begin{equation}\label{eqn-sfn}
  s_\mu(\x) = \frac{a_{\mu+\delta}(\x)}{a_\delta(\x)} 
	   = \frac{\ds \det_{1\leq i,j\leq n} \big(\, x_i^{\mu_j+n+1-j} \,\big)}{\ds\det_{1\leq i,j\leq n} \big(\, x_i^{n+1-j} \,\big)} \,.
\end{equation}
It is easy to see that the Schur function denominator $a_{\delta}(\x)$, as defined here, is given by 
\begin{equation}\label{eqn-adelta}
  a_{\delta}(\x) =  \prod_{i=1}^n\,x_i\, \prod_{1\leq i<j\leq n} (x_i-x_j)\,. 
\end{equation}
Moreover in the expansion of the numerator $a_{\mu+\delta}(\x)$ as a signed sum of monomials of the form 
$\x^\kappa=x_1^{\kappa_1}x_2^{\kappa_2}\cdots x_n^{\kappa}$, the parts of $\kappa$ are
necessarily distinct and there is only one term for which $\kappa$ is 
a partition, namely the leading term $\x^{\mu+\delta}$. It follows that for any strict partition $\lambda$ 
\begin{equation}\label{eqn-alambda}
  [\x^\lambda]\ a_{\mu+\delta}(\x) = \delta_{\lambda,\mu+\delta}\,.
\end{equation}

The decomposition of Schur function products and quotients take the form 
\begin{equation}\label{eqn-LR}
   s_\mu(\x) \ s_\nu(\x)  = \sum_{\lambda\in{\cal P}}\ c_{\mu\nu}^\lambda \ s_\lambda(\x)
	\quad\mbox{and}\quad 
	 s_{\lambda/\mu}(\x) = \sum_{\nu\in{\cal P}}\ c_{\mu\nu}^\lambda \ s_\nu(\x)\,,
\end{equation}
where the coefficients $c_{\mu\nu}^\lambda$ are non-negative integers determined by the famous 
Littlewood-Richardson rule, and the quotient is usually referred to as a skew Schur function. 

Finally, there are three generating functions of relevance here, namely the Cauchy formula~\cite{Macdonald}
\begin{equation}\label{eqn-Cauchy}
    \prod_{i,j=1}^{m,n} \frac{1}{1-x_i y_j} = \sum_{\sigma\in{\cal P}} s_\sigma(\x)\ s_\sigma(\y)
\end{equation}
for any $\x=(x_1,x_2,\ldots,x_m)$ and $\y=(y_1,y_2,\ldots,y_n)$,
and the formulae~\cite{Littlewood}
\begin{equation}\label{eqn-ACalphagamma}
\begin{array}{rcl}
   \ds \prod_{1\leq i<j\leq n}  (1+x_ix_j)  &=& \ds \sum_{\alpha\in{\cal A}={\cal P}_{1}}\ s_\alpha(\x)\,; \cr\cr
   \ds \prod_{1\leq i\leq j\leq n}  (1+x_ix_j)  &=& \ds \sum_{\gamma\in{\cal C}={\cal P}_{1}}\ s_\gamma(\x)\,; \cr\cr
\end{array}	
\end{equation}
where ${\cal P}_t$ is the set of all the partitions which in Frobenius notation~\cite{Macdonald} 
are such that the arm length minus the corresponding leg length is $t$ for any integer $t$.

Now we turn attention to the orthogonal group $SO(2n+1)$ and its corresponding Lie algebra $B_n\cong so(2n+1)$.
For each $n\in\N$ and each partition $\mu$ of length $\ell(\mu)\leq n$ 
there exists an irreducible representation $V_{SO(2n+1)}^\mu$ of $SO(2n+1)$ of highest weight $\mu$, 
whose character evaluated for a group element with eigenvalues $(1,\x,\ov{\x})=(1,x_1,\ldots,x_n,\ov{x}_1,\ldots,\ov{x}_n)$ is given by
\begin{equation}
  \ch V_{SO(2n+1)}^\mu  = so_\mu(1,\x,\ov{\x})\,,
\end{equation}
where
\begin{equation}
 so_\mu(1,\x,\ov{\x})
	   = \frac{\ds \det_{1\leq i,j\leq n} \big(\, x_i^{\mu_j+n-j+1/2}- x_i^{-(\mu_j+n-j+1/2)}  \,\big)}
		         {\ds\det_{1\leq i,j\leq n} \big(\, x_i^{n-j+1/2} - x_i^{-(n-j+1/2)}\,\big)} \,.
\end{equation}
It is possible, and convenient, to extend this definition so as to encompass what is sometimes
referred to as a universal character for the orthogonal group. When expressed in terms of Schur
functions this takes the form
\begin{equation}\label{eqn-sochar}
  so_\mu(\z) = \sum_{\gamma\in{\cal C}} \ (-1)^{|\gamma|/2}\ s_{\mu/\gamma}(\z)
\end{equation}
where now $\z=(z_1,z_2,\ldots,z_{2n+1)})$ involves $2n+1$ parameters that have to be specialised 
to recover the actual group character. It is this universal character that will be shown to emerge 
naturally in our later analysis.

To make contact with Weyl's denominator formula it should be recalled that the set of positive roots 
of the Lie algebra $o(2n+1)$ of the orthogonal group $SO(2n+1)$ is given by
\begin{equation}\label{eqn-roots}
    \Delta^+=\{\epsilon_i\,|\,1\leq i\leq n\}\cup\{\epsilon_i\pm\epsilon_j\,|\, 1\leq i<j\leq n\}\,.
\end{equation}
It follows that Weyl's denominator takes the form
\begin{equation}\label{eqn-odenom}
  \prod_{\alpha\in\Delta^+} (1 - e^{-\alpha}) = \prod_{i=1}^n (1-x_i) \prod_{1\leq i<j\leq n} (1-x_ix_j)(1-x_ix_j^{-1})
\end{equation}
where $x_i=e^{-\epsilon_i}$ for $i=1,2,\ldots,n$. It is deformations of this formula through 
the introduction of additional parameters that we will meet in what follows.

\section{Tokuyama type factorisation}
\label{sec:Tok}

The significance of Lemma~\ref{lem-deth} lies in the fact that it allows us to proceed smoothly
from Theorem~\ref{the-Pdet} to our main result, namely the following
Tokuyama type factorisation theorem involving group characters. 

\begin{Theorem}\label{the-result2}
For all $n\in\N$ let $\lambda=\mu+\delta$ where $\delta=(n,n-1,\ldots,1)$ 
with $\mu$ a partition of length $\ell(\mu)\leq n$. 
Then for $\z=(x_1,x_2,\ldots,x_n,z_0,\ov{y}_n,\ldots,\ov{y}_2,\ov{y}_1)$
and the contributions to the weight $\wgt(P)$ of each primed shifted tableau 
$P\in{\cal P}^\lambda$ as specified in the tabulation (\ref{tab-Pwgt}) we have 
\begin{equation}\label{eqn-Tok}
   \sum_{P\in{\cal P}^{\lambda}(2n+1)} \wgt(P) = \sum_{P\in{\cal P}^\delta(2n+1)} \wgt(P)\  
	                                      \sum_{\gamma\in{\cal C}}  s_{\mu/\gamma} (\z)
\end{equation}
with
\begin{equation}\label{eqn-Zfact}
    \sum_{P\in{\cal P}^\delta(2n+1)} \wgt(P) = 
			\prod_{i=1}^n x_i^{-(n-i)}\ \prod_{i=1}^n (1+z_0x_i) \  
			\prod_{1\leq i<j\leq n} (1+ x_ix_j)(1+x_i\ov{y}_j)\,.
\end{equation}
\end{Theorem}

\begin{proof}
Our previous results taken together imply that
\begin{equation}\label{eqn-proof}
\begin{array}{rcl}
\ds \sum_{P\in{\cal P}^{\lambda}(2n+1)}\!\!\! \wgt(P)
&=&\ds [\q^\lambda]\ \det(\,h_{k,\ell}(q_\ell)\,) 
     = [q^\lambda]\ K(\z,\q)\ Q(\q)\ Z(\z) \cr
&=&\ds Z(\z)\ [q^\lambda] \ \sum_{\sigma\in{\cal P}} s_\sigma(\z) s_\sigma(\q)\ 
		                     a_\delta(q)\ \sum_{\gamma\in{\cal C}} s_\gamma(\q) \cr
&=& \ds Z(\z)\ [q^\lambda]\ \sum_{\tau\in{\cal P}}\ \sum_{\gamma\in{\cal C}} s_{\tau/\gamma}(\z) \ a_\delta(\q)\ s_{\tau}(\q) \cr
&=&\ds  Z(\z)\ \sum_{\tau\in{\cal P}}\ \sum_{\gamma\in{\cal C}} s_{\tau/\gamma}(\z)\ [q^\lambda] \  a_{\tau+\delta}(\q) \cr
&=& \ds  Z(\z)\ \sum_{\tau\in{\cal P}}\ \sum_{\gamma\in{\cal C}} s_{\tau/\gamma}(\z)\ \delta_{\mu+\delta,\tau+\delta}
        =  Z(\z)\ \sum_{\gamma\in{\cal C}}  s_{\mu/\gamma} (\z)\,. 
\end{array}
\end{equation}
Here the successive steps may be justified as follows. The first step just corresponds to noting 
that, for the given $\z$ and $\q$, (\ref{eqn-ZKQ}) and (\ref{eqn-Cauchy}) imply
\begin{equation}
 K(\z,\q)=\prod_{i,j=1}^{2n+1,n} \frac{1}{1-z_iq_j}= \sum_{\sigma\in{\cal P}}\ s_\sigma(\z)\ s_\sigma(\q)\,,
\end{equation}
while (\ref{eqn-ZKQ}), (\ref{eqn-adelta}) and (\ref{eqn-ACalphagamma}) imply
$Q(\q) = \sum_{\gamma\in{\cal C}} a_\delta(\q)\,s_\gamma(\q)$. 
The next step is a consequence of the dual nature of the two formulae of (\ref{eqn-LR}) for Schur function products and quotients.
These imply first that $s_\sigma(\q)\,s_\gamma(\q)=\sum_{\tau\in\cal{P}} c_{\sigma\gamma}^\tau\,s_\tau(\q)$ and then
that $\sum_{\sigma\in\cal{P}} c_{\sigma\gamma}^\tau\,s_\tau(\z)=s_{\tau/\gamma}(\z)$.
The final three steps arise from the Schur function definition (\ref{eqn-sfn}) giving 
$a_\delta(\q)\,s_\tau(\q)=a_{\tau+\delta}(\q)$, followed by the use of (\ref{eqn-alambda}) 
and the fact that by hypothesis $\lambda=\mu+\delta$.

Then setting $\mu=0$ in (\ref{eqn-proof}) gives 
\begin{equation}
  \sum_{P\in{\cal P}^{\delta}(2n+1)} \wgt(P) = Z(\z)\,.
\end{equation}
Substituting this back into (\ref{eqn-proof}) gives (\ref{eqn-Tok}), while the explicit form
of $Z(\z)$ in (\ref{eqn-ZKQ}) gives (\ref{eqn-Zfact}), thereby completing the proof of
the Theorem~\ref{the-result2}.  \end{proof}

If the entries in the primed shifted tableaux $P$ are restricted to those in the alphabet
\begin{equation}
       1'<1<2'<2<\cdots<n'<n<\ov{n}'<\ov{n}<\cdots<\ov{2}'<\ov{2}<\ov{1}'<\ov{1}
\end{equation}
so that entries $0$ and $0'$ are no longer allowed, then we arrive at:
\begin{Theorem}\label{the-result3}
For all $n\in\N$ let $\lambda=\mu+\delta$ where $\delta=(n,n-1,\ldots,1)$ 
with $\mu$ a partition of length $\ell(\mu)\leq n$. 
Then for $\z=(x_1,x_2,\ldots,x_n,1,\ov{y}_n,\ldots,\ov{y}_2,\ov{y}_1)$
and the contributions to the weight $\wgt(P)$ of each primed shifted tableau 
$P\in{\cal P}^\lambda$ as specified in the following table
\begin{equation}\label{tab-Pwgt-2n}
\begin{array}{|l|l|}
\hline
\hbox{Diagonal weights}&\hbox{Off-diagonal weights} \cr
\hline
 k\mapsto (-1)^{n-k+1}x_k \prod_{j>k} (x_j\ov{y}_j) &k\mapsto I\,x_k\cr
 \ov{k}\mapsto 1    &k'\mapsto -I\,y_k\cr
                    &\ov{k}\mapsto I\,\ov{y}_k\cr
                    &\ov{k}'\mapsto -I\,\ov{x}_k\cr
\hline
\end{array}
\end{equation}
we have 
\begin{equation}\label{eqn-Pwgt-lambda-bn}
   \sum_{P\in{\cal P}^{\lambda}(2n)} \wgt(P) = \sum_{P\in{\cal P}^\delta(2n)} \wgt(P)\  
	                                     I^{|\mu|} \ \sum_{\gamma\in{\cal C}} (-1)^{|\gamma|/2} s_{\mu/\gamma} (\z)
\end{equation}
with
\begin{equation}\label{eqn-Pwgt-delta-bn}
    \sum_{P\in{\cal P}^\delta(2n)} \wgt(P) = 
			\prod_{i=1}^n (I\,x_i)^{-(n-i)}\ \prod_{i=1}^n (1-x_i) \  
			\prod_{1\leq i<j\leq n} (1-x_ix_j)(1-x_i\ov{y}_j)\,.
\end{equation}
\end{Theorem}

\begin{proof}
Entries $0$ and $0'$ are excluded in the sum over lattice paths associated with $P$ merely by setting $z_0=I=\sqrt{-1}$
since in this case the factor $(1+\ov{z}_0q_\ell)/(1-z_0q_\ell)$ appearing in $f_{k,l}(q_\ell)$ in (\ref{eqn-f}) reduces to $1$.
The result then follows from Theorem~\ref{the-result2} under the additional substitutions $x_i=I\,x_i$ and $\ov{y}_i=I\,\ov{y}_i$
for $i=1,2,\ldots,n$.
\end{proof}

\section{Corollaries}
\label{sec:corollaries}

Thanks to Lemma~\ref{lem-AP} the result Theorem~\ref{the-result2} 
can be restated immediately in terms of ASMs as 
\begin{Corollary}\label{cor-Tok-bnprime}
For all $n\in\N$ let $\lambda=\mu+\delta$ where $\delta=(n,n-1,\ldots,1)$ 
with $\mu$ a partition of length $\ell(\mu)\leq n$. 
Then for $\z=(x_1,x_2,\ldots,x_n,z_0,\ov{y}_n,\ldots,\ov{y}_2,\ov{y}_1)$
and the contributions to the weight $\wgt(A)$ of each $\lambda$-HTSASM 
$A\in{\cal B'}^\lambda_n$ as specified in the tabulation (\ref{tab-Cwgt}) of Definition~\ref{def-Awgt} we have 
\begin{equation}\label{eqn-bnprime-lambda}
   \sum_{A\in{\cal B'}^{\lambda}_n} \wgt(A) = \sum_{A\in{\cal B'}^\delta_n} \wgt(A)\  
	                                      \sum_{\gamma\in{\cal C}}  s_{\mu/\gamma} (\z)
\end{equation}
with
\begin{equation}\label{eqn-bnprime-delta}
    \sum_{A\in{\cal B'}_n^\delta} \wgt(A) = 
			\prod_{i=1}^n (1+z_0x_i) \  
			\prod_{1\leq i<j\leq n} (1+ x_ix_j)(1+x_i\ov{y}_j)\,.
\end{equation}
\end{Corollary} 

It can be seen that (\ref{eqn-bnprime-lambda}) is a direct generalisation of Tokuyama's identity (\ref{eqn-tok-lambda}) from the 
general linear group $GL(n)$ to the orthogonal group $SO(2n+1)$ in which 
not only is Weyl's denominator deformed 
from (\ref{eqn-odenom}) to (\ref{eqn-bnprime-delta}) through the dependence on both $x_i$ and $y_i$, but so is the
corresponding character which is expressed here as a sum of Schur functions that can be thought of as a deformation
of the expression for the universal orthogonal group character $so_\mu(\z)$ specified in (\ref{eqn-sochar}).

We have concentrated so far on the $\lambda$-HTSASMs of type ${\cal B'}^\lambda_n$ with an odd number of rows, but our results also 
encompass those of type ${\cal B}^\lambda_n$ with an even number of rows. The latter are in bijective correspondence with the former
restricted to those cases in which the central row has no non-zero entries to the right of the central $1$. The bijection
just involves the removal of both the central row and the central column.
This is illustrated in the case $n=3$ and $\lambda=\delta=(3,2,1)$ by
\begin{equation}
\begin{array}{ccccc}
\mbox{Type $A\in B'_3$:}&&\mbox{Type $A\in B_3$:}&&T\cr\cr
\left[
\begin{array}{ccccccc}
0 & 0 & 0 & 0 & 1 & 0 & 0 \\
0 & 1 & 0 & 0 &\ov1 & 0 & 1 \\
0 & 0 & 1 & 0 & 0 & 0 & 0\\
0 & 0 & 0 & 1 & 0 & 0 & 0  \\
0 & 0 & 0 & 0 & 1 & 0 & 0  \\
1 & 0 &\ov1 & 0 & 0 & 1 & 0 \\
0 & 0 & 1 & 0 & 0 & 0 & 0 \\
\end{array}
\right]
&\Leftrightarrow&
\left[
\begin{array}{cccccc}
0 & 0 & 0  & 1 & 0 & 0 \\
0 & 1 & 0 &\ov1 & 0 & 1 \\
0 & 0 & 1 &  0 & 0 & 0\\
0 & 0 & 0 & 1 & 0 & 0  \\
1 & 0 &\ov1 & 0 & 1 & 0 \\
0 & 0 & 1 & 0 & 0 & 0 \\
\end{array}
\right]
&\Leftrightarrow&
{\vcenter
  {\offinterlineskip
 \halign{&\mystrut\vrule#&\mybox{\hss$#$\hss}\cr
  \hr{7}\cr
                      &1&&2 &&2   &\cr     
  \hr{7}\cr
                      \omit& &&\ov3&&  \ov2 &\cr          
  \nr{2}&\hr{5}\cr
                      \omit& &\omit& &&\ov2   &\cr      
  \nr{4}&\hr{3}\cr  
}}}	
\end{array}
\label{asm-B'3B3}
\end{equation}
The corresponding shifted tableau $T$ is the same in both cases, and as shown above contains no entries $0$.
This observation remains true for all $n$ and all strict partitions $\lambda$.

To eliminate contributions to our various sums over weights, $\wgt(A)$, from those $A\in{\cal B'}^\lambda_n$ that are 
not of the above type it suffices to note that they will necessarily contain at least one entry $\ov1=-1$ 
in the central row, and correspondingly an entry $\NS$ in row $i=n+1$ of the associated CPM. Such an entry gives rise to 
a factor $(z_0+\ov{z}_0)$ in the expression for $\wgt(A)$, as can be seen from the weight assignments given 
in (\ref{tab-Cwgt}). The required elimination is therefore automatically accomplished by 
setting $z_{0}=I=\sqrt{-1}$. 
If one does this, and then for convenience replaces $x_i$ and $\ov{y}_i$ by $Ix_i$ and $I\ov{y}_i$, respectively,
one arrives at the following corollary that could also be obtained from Theorem~\ref{the-result3}:

\begin{Corollary}\label{cor-Tok-bn}
For $n\in\N$ let $\x=(x_1,x_2,\ldots,x_{n})$, $\y=(y_1,y_2,\ldots,y_n)$
and $\z=(x_1,\ldots,x_n,1,\ov{y}_n,\ldots,\ov{y}_1)$. 
For any $\lambda=\mu+\delta$ with $\delta=(n,n-1,\ldots,1)$ and $\mu$ a partition of length $\ell(\mu)\leq n$
and $\lambda_1=m\geq n$, let ${\cal B}^\lambda_n$ be the set of all $2n\times m$ 
$\lambda$-HTSASMs. For each $A\in{\cal B}^\lambda_n$ let $C$ be the corresponding $2n\times m$
CPM with matrix elements $c_{ij}$ and let
\begin{equation}\label{eqn-Awgt-bn}
   \wgt(A) = (-1)^{n(n-1)/2} \prod_{i=1}^n\ x_i^{n-i}\ (-1)^{L_{i+1}} (x_i\ov{y}_i)^{L_i} \ x_i^{L_{i+1}-L_i}\  \prod_{i=1}^{2n+1}\prod_{j=1}^m \wgt(c_{ij})
\end{equation}
where $\wgt(c_{ij})$ is as tabulated below
\begin{equation}\label{tab-Cwgt-bn}
\begin{array}{|l|l|l|}
\hline
\hbox{Entry}&i\leq n&i>n\cr
\hbox{at $(i,j)$}&j\geq1&j\geq1\cr
\hline
\WE&1&1\cr
\NS&x_i-y_i&\ov{y}_{2n+1-i}-\ov{x}_{2n+1-i}\cr
\NE&1&1\cr
\SE&1&1\cr
\NW&-y_i&-\ov{x}_{2n+1-i}\cr
\SW&x_i&\ov{y}_{2n+1-i}\cr
\hline
\end{array}
\end{equation}
Then 
\begin{equation}\label{eqn-Tok-bn-lambda}
      \sum_{A\in{\cal B}^\lambda_n} \ \wgt(A) = \sum_{A\in{\cal B}^\delta_n} \  \wgt(A) \ \ \sum_{\gamma\in{\cal C}} (-1)^{|\gamma|/2} s_{\mu/\gamma} (\z)\,,
\end{equation}
where
\begin{equation}\label{eqn-Tok-bn-delta}
			\sum_{A\in{\cal B}^\delta_n} \wgt(A) = 
			\prod_{i=1}^n (1-x_i)\ \prod_{1\leq i<j\leq n} (1-x_ix_j)(1-x_i\ov{y}_j)\,.
\end{equation} 
\end{Corollary}

This time while we have a minor variation of the deformation of Weyl's denominator formula (\ref{eqn-odenom}) obtained
in the ${\cal B'}_n$ case, the deformed character that arises in this ${\cal B}_n$ case is none other than the universal 
orthogonal group character $so_\mu(\z)$ defined in (\ref{eqn-sochar}).

To recover all the results stated in Section~\ref{sec:ASMweights} it is only necessary to specialise appropriately
the parameters appearing in Corollary~\ref{cor-Tok-bnprime}. The required specialisations are as follows:

\begin{equation}\label{tab-theorems}
\begin{array}{|l|l|l|l|l|}
\hline
\hbox{Parameters}&\lambda-\hbox{HTSASMs}&x_i&z_0&\ov{y}_i\cr
\hline
\hbox{Theorem \ref{the-result1}}&A\in{\cal B'}^\lambda_n&s_ix_i&z_0&t_i\ov{y}_i\cr
\hbox{Theorem \ref{the-Tabony}}&A\in{\cal B}^\lambda_n&It_ix_i&I&It_i\ov{x}_i\cr
\hbox{Theorem \ref{the-Simpson}}&A\in{\cal B'}^\delta_n&tx_i&t&t\ov{x}_i\cr
\hbox{Theorem \ref{the-Okada}}&A\in{\cal B}^\delta_n&Itx_i&It&It\ov{x}_i\cr
\hline
\end{array}
\end{equation}

Undertaking the first two specialisations immediately yields the hitherto unknown factors appearing in 
Theorem \ref{the-result1} and Theorem \ref{the-Tabony}:
\begin{equation}
\begin{array}{rcl}
\Phi_{B'_n}^\mu(\z)&=&  \sum_{\gamma\in{\cal C}} s_{\mu/\gamma} (\z) \cr\cr
\Phi_{B_n}^\mu(\z)&=& I^{|\mu|} \sum_{\gamma\in{\cal C}} (-1)^{|\gamma|/2} s_{\mu/\gamma} (\z)= I^{|\mu|}\, so_\mu(x_1,\ldots,x_n,1,\ov{y}_n,\ldots,\ov{y}_1),
\end{array}
\end{equation}

\section{Addendum}
\label{sec:BS}

While this work was in progress a further private communication received from Brubaker prompted
us to consider a final generalisation of our deformation parameters to accommodate  
the universal parameterisation of Brubaker and Schultz~\cite{BrubakerSchultz}. 
This takes the form of the assignment of 
universal weights $\wgt(c_{ij})$ to compass point matrix elements shown on the left:
\begin{equation}\label{tab-BSuniversal-wgts}
\begin{array}{|l|l|l|l|l|}
\hline
\hbox{Entry}&\hbox{Universal}&i<n+1&i=n&i>n+1\cr
\hbox{at $(i,j)$}&\hbox{all $i$}&j\geq1&j\geq1&j\geq1\cr
\hline
\WE&c_2^{(i)}&1&1&1\cr
\NS&c_1^{(i)}&a_1^{(i)}a_2^{(i)}+b_1^{(i)}b_2^{(i)}&(a_0^{(0)})^2+(b_0^{(0)})^2&a_1^{(\bar{\imath})}a_2^{(\bar{\imath})}+b_1^{(\ov{i})}b_2^{(\ov{i})}\cr
\NE&b_2^{(i)}&b_2^{(i)}&b_0^{(0)}&b_1^{(\bar{\imath})}\cr
\SE&a_2^{(i)}&a_2^{(i)}&a_0^{(0)}&a_1^{(\bar{\imath})}\cr
\NW&a_1^{(i)}&a_1^{(i)}&a_0^{(0)}&a_2^{(\bar{\imath})}\cr
\SW&b_1^{(i)}&b_1^{(i)}&b_0^{(0)}&b_2^{(\bar{\imath})}\cr
\hline
\end{array}
\end{equation}
We have set this alongside those weights adopted by Brubaker and Schultz~\cite{BrubakerSchultz} that are appropriate to the $B'_n$ case as dictated by  
the free-fermion condition
\begin{equation}
   c_1^{(i)}c_2^{(i)} = a_1^{(i)}a_2^{(i)}+b_1^{(i)}b_2^{(i)} \,,
\end{equation}
together with the symmetry conditions appropriate to $U$-turn $\lambda$-HTSASMs
\begin{equation}
   a_1^{(i)}=a_2^{(\bar{\imath})},\quad  b_1^{(i)}=b_2^{(\bar{\imath})}, \quad c_1^{(i)}=c_1^{(\bar{\imath})} ,\quad c_2^{(i)}=c_2^{(\bar{\imath})} \,,
\end{equation}
and the normalisation condition
\begin{equation}
   c_2^{(i)}=1 \,,
\end{equation}
all for $i=1,2,\ldots,2n+1$. The notation used here is such
that $\bar{\imath}=2n+2-i$ for all $i=1,2,\ldots,2n+1$, $a_0^{(0)}=a_1^{(n+1)}=a_2^{(n+1)}$ and $b_0^{(0)}=b_1^{(n+1)}=b_2^{(n+1)}$.

Greatly encouraged by correspondence with Brubaker, we were led to the following extension
of the result appropriate to ${\cal B}_\ast^\lambda$ in 
the long list of factorisation results appearing in Theorem~4.5 of~\cite{BrubakerSchultz}:
\begin{Theorem}\label{the-BS-Tok}
For all $n\in\N$ let $\lambda=\mu+\delta$ where $\delta=(n,n-1,\ldots,1)$ 
with $\mu$ a partition of length $\ell(\mu)\leq n$ and $m=\lambda_1$, and let
\begin{equation}\label{eqn-BS-z}
\z=\left(b_1^{(1)}/a_2^{(1)},\ldots,b_1^{(n)}/a_2^{(n)},b_0^{(0)}/a_0^{(0)},b_2^{(n)}/a_1^{(n)},\ldots,b_2^{(1)}/a_1^{(1)} \right)
\end{equation}
Then for each $\lambda$-HTSASM $A\in{\cal B'}^\lambda_n$ let 
\begin{equation}\label{eqn-BSwgtA}
  \wgt(A) =  w_0w_1\ \prod_{i=1}^n \left(\frac{b_1^{(i)}}{a_2^{(i)}}\right)^{\!n-i}\!\!    
			                                \left(\frac{b_0^{(0)}b_1^{(i)}}{a_0^{(0)}a_2^{(i)}}\right)^{\!L_{i+1}-L_i}\!\! 
																	    \left(\frac{b_1^{(i)}b_2^{(i)}}{a_2^{(i)}a_1^{(i)}}\right)^{\!L_i} \                                                                                       \prod_{i=1}^{2n+1}\prod_{j=1}^m \wgt(c_{ij})
\end{equation}
with 
\begin{equation}\label{eqn-w0w1}
   w_0 = \left( a_0 \prod_{i=1}^n a_1^{(i)} a_2^{(i)} \right)^{m-n} \,,~~~
	 w_1 =  a_0^n \prod_{i=1}^n (a_1^{(i)})^{i-1} (a_2^{(i)})^{2n-i} 
\end{equation}
and $\wgt(c_{ij})$ specified by
\begin{equation}\label{eqn-BSwgtC}.
\begin{array}{|l|l|l|l|l|}
\hline
\hbox{Entry}&i<n+1&i=n&i>n+1\cr
\hbox{at $(i,j)$}&j\geq1&j\geq1&j\geq1\cr
\hline
\WE&1&1&1\cr
\NS&b_1^{(i)}/a_2^{(i)}+a_1^{(i)}/b_2^{(i)}&b_0^{(0)}/a_0^{(0)}+a_0^{(0)}/b_0^{(0)}&b_2^{(\bar{\imath})}/a_1^{(\bar{\imath})}+a_2^{(\bar{\imath})}/b_1^{(\bar{\imath})}\cr
\NE&1&1&1\cr
\SE&1&1&1\cr
\NW&a_1^{(i)}/b_2^{(i)}&a_0^{(0)}/b_0^{(0)}&a_2^{(\bar{\imath})}/b_1^{(\bar{\imath})}\cr
\SW&b_1^{(i)}/a_2^{(i)}&b_0^{(0)}/a_0^{(0)}&b_2^{(\bar{\imath})}/a_1^{(\bar{\imath})}\cr
\hline
\end{array}
\end{equation}

Then 
\begin{equation}\label{eqn-BS-lambda} 
   \sum_{A\in{\cal B'}_n^\lambda} \wgt(A) = \sum_{A\in{\cal B'}_n^\delta} \wgt(A)\  w_0  \sum_{\gamma\in{\cal C}}  s_{\mu/\gamma} (\z)
\end{equation}
and, as derived by Brubaker and Schultz~\cite{BrubakerSchultz},
\begin{equation}\label{eqn-BS-delta} 
    \sum_{A\in{\cal B'}_n^\delta} \wgt(A) =
			\prod_{i=1}^n (a_0^{(0)}a_2^{(i)}+b_0^{(0)}b_1^{(i)}) 
			\!\!\prod_{1\leq i<j\leq n}\! (a_2^{(i)}a_1^{(j)}+b_1^{(i)}b_2^{(j)})  (a_2^{(i)}a_2^{(j)}+b_1^{(i)}b_1^{(j)}) \,.
\end{equation}
\end{Theorem}

\begin{proof}
First it should be noted that the passage from the Brubaker and Schultz weighting specified in the final three columns of (\ref{tab-BSuniversal-wgts})
to that of (\ref{eqn-BSwgtA})--(\ref{eqn-BSwgtC}) is accomplished through the use of the crucial compass point matrix identities
(\ref{eqn-wenese}) and the symmetry condition (\ref{eqn-Lsymmetry}). Then comparison of (\ref{eqn-BSwgtA}) 
and the tabulation (\ref{eqn-BSwgtC}) with (\ref{eqn-Awgt}) and the tabulation (\ref{tab-Cwgt}), respectively, 
of Definition~\ref{def-Awgt} in Section~\ref{sec:PSTandLPwgts} shows that with the identification
\begin{equation}\label{eqn-xyz-ab}
  z_0 = b_0^{(0)}/a_0^{(0)}  \quad\hbox{and}\quad
   x_i = b_1^{(i)}/a_2^{(i)},~~ y_i = a_1^{(i)}/b_2^{(i)} \quad\hbox{for $i=1,2,\ldots,n$}
\end{equation}
the hypotheses of Theorem~\ref{the-BS-Tok} coincide with those of Corollary~\ref{cor-Tok-bnprime} apart from the
inclusion of additional scaling factors $w_0$ and $w_1$. 
The factor $w_0$ emerges unscathed on the right hand side of (\ref{eqn-BS-lambda})
since $w_0=1$ in the case $m=n$ that arises when $\lambda=\delta$.
The other additional factor $w_1$ is just what is required to
rewrite the outcome (\ref{eqn-bnprime-delta}) of Corollary~\ref{cor-Tok-bnprime}, which was deliberately
normalised so as to have leading term $1$, in the form (\ref{eqn-BS-delta}) which has leading term $w_1$.
The final observation is that $\z$ as defined in Corollary~\ref{cor-Tok-bnprime} reduces to $\z$ as specified in
(\ref{eqn-BS-z}) under the identification (\ref{eqn-xyz-ab}). \end{proof}

It should be noted that this Theorem~\ref{the-BS-Tok} serves to determine explicitly the ratio of partition functions
appearing in the ${\cal B}^\lambda_\ast$ model of Brubaker and Schultz~\cite{BrubakerSchultz}. This ratio takes the form
\begin{equation}
  {\cal Z}({\cal B}^\lambda_\ast) \big/  {\cal Z}({\cal B}^\delta_\ast)  = w_0  \sum_{\gamma\in{\cal C}}  s_{\mu/\gamma} (\z)
\end{equation}
with $w_0$ and $\z$ defined in terms of the Brubaker-Schultz parameters by (\ref{eqn-w0w1}) and (\ref{eqn-BS-z}), respectively.

Finally, it may be noted that on comparing our formalism with that of Brubaker and Schultz, our labelling of the vertices in the square-ice configuration model are identical, but the configurations themselves are different: our diagrams are upside down and left-right reversed compared with those of \cite{BrubakerSchultz}. This implies that the compass point matrix elements differ by the mutual interchange of $\NE$ and $\SW$, and of $\SE$ and $\NW$. However,
with the relabelling of rows from top to bottom, the $U$-turn symmetry conditions for $\lambda$-HTSASMs are just what is required to ensure that our weighting used in Theorem~\ref{the-BS-Tok} coincides with that of Brubaker and Schultz.
\bigskip \bigskip

\noindent{\bf Acknowledgements} \medskip

We have made extensive use of Maple to generate examples and verify results.
We are grateful to the organizers of the Banff International Research Station for Mathematical Innovation and Discovery (BIRS) 2010 5-day workshop, ``Whittaker Functions, Crystal Bases, and Quantum Groups'' where the first author (AMH) first learned of this problem. She also acknowledges the hospitality of the Department of Combinatorics and Optimization, University of Waterloo during 2013-4, and the support of a
Discovery Grant from the Natural Sciences and Engineering Research Council of
Canada (NSERC). The second author (RCK) is grateful for the hospitality and financial support
extended to him while visiting both Wilfrid Laurier University, Waterloo and the Center for Combinatorics at Nankai
University, Tianjin where some of this work was carried out. 
Both authors are pleased to acknowledge support from ICERM for their participation in a workshop of the 
Semester Program ''Automorphic Forms, Combinatorial Representation Theory and Multiple Dirichlet Series''.
Finally, the authors are extremely grateful to Professor Brubaker both for his general encouragement and for
making so much of his own work and that of his colleagues so readily available.   
\bigskip



\bigskip
\section{Appendix}
\medskip


\renewcommand{\theequation}{A.\arabic{equation}}

Our aim here is to provide an independent proof of Lemma~\ref{lem-detm} which for convenience we repeat here as
\begin{appxlem}\label{lem-detm-A}
For $n\in\N$, let $\q=(q_1,q_2,\ldots,q_n)$, $\a=(a_1,a_2,\ldots,a_n)$,
$\b=(b_1,b_2,\ldots,b_n)$ and $\c=(c_1,c_2,\ldots,c_n,c_{n+1})$
be four sequences of indeterminates, with $q_i\neq0$ and $\ov{q}_i=q_i^{-1}$ for $i=1,2,\ldots,n$,
and set $m_{k,\ell}(q_\ell)=\tilde{g}_{k,\ell}(q_\ell)-\tilde{g}_{k,\ell}(-\ov{q}_\ell)$
where
\begin{equation}\label{eqn-tg-A}
\tilde{g}_{k,\ell}(q_\ell) = q_\ell^{-n}\ \prod_{i=1}^{n+1}\,(1+c_iq_\ell)\ \prod_{i=1}^{k-1}\,(1-b_iq_\ell)\,  \prod_{i=k+1}^n (1+a_iq_\ell) \,.
\end{equation}
Then
\begin{equation}\label{eqn-detm-A}
\begin{array}{l}
\det_{1\leq k,\ell\leq n}  \big(\, m_{k,\ell}(q_\ell)\,\big) \cr
\ds = \prod_{1\leq i<j\leq n}\!\!\!\! (b_i+a_j)\!\! \prod_{1\leq i<j\leq n+1}\!\!\!\!\! (1+c_ic_j)\ 
\prod_{i=1}^n\ q_i^{-n}(1+q_i^2)\!\!\prod_{1\leq i<j\leq n}\!\!\! (q_i-q_j)(1+q_iq_j)\,.
\end{array}
\end{equation}
\end{appxlem}

\begin{proof}
Our starting point is the required determinant expressed in the form 
\begin{equation}
\begin{array}{l}
\ds \det_{1\leq k,\ell\leq n}  \big(\, m_{k,\ell}(q_\ell)\,\big) 
= \det_{1\leq k,\ell\leq n} \left( L(\c,q_\ell)\ \ov{q}_\ell^{n} \prod_{i=1}^{k-1} (1-b_iq_\ell) \prod_{i=k+1}^n (1+a_iq_\ell) \right. \cr
\ds ~~~~~~~~~~~~~~~~~~~~~~~~~~~~~~~~~~~\left. - L(\c,-\ov{q}_\ell)\ (-q_\ell)^{n} \prod_{i=1}^{k-1} (1+b_i\ov{q}_\ell) \prod_{i=k+1}^n (1-a_i\ov{q}_\ell) \right)\cr
\end{array}
\end{equation}
where
\begin{equation}
L(\c,q_\ell)= \prod_{i=1}^{n+1} (1+c_iq_\ell) \,.
\end{equation}

We first separate out the dependence on the various $a_i$ and $b_i$ as follows.
Subtracting row $k+1$ from row $k$ yields a factor $(b_k+a_{k+1})q_\ell$ from the
left hand term and a factor $(b_k+a_{k+1})/(-q_\ell)$ from the right hand term for $k=1,2,\ldots,n-1$.
We can take out the common factor $b_k+a_{k+1}$ and reduce the powers of $\ov{q}_\ell$ and $(-q)_\ell$ from
$n$ to $n-1$. Then subtracting row $k+1$ from row $k$ yields
a common factor $b_k+a_{k+2}$ for $k=1,2,\ldots,n-2$ while the powers $\ov{q}_\ell$ and $(-q)_\ell$ are reduced
from $n-1$ to $n-2$. Continuing in this way one finds that
\begin{equation}
\begin{array}{l}
\ds \det_{1\leq k,\ell\leq n}\left(m_{k\ell}(q_\ell)\right) = \prod_{1\leq i<j\leq n}\!\! (b_i+a_j)\ \cr
\ds \times \det_{1\leq k,\ell\leq n} 
\left( L(\c,q_\ell)\ \ov{q}_\ell^{k}\ \prod_{i=1}^{k-1} (1-b_iq_\ell)\  -  L(\c,-\ov{q}_\ell)\ (-q_\ell)^{k}\ \prod_{i=1}^{k-1} (1+b_i\ov{q}_\ell) \right)\,.\cr
\end{array}
\end{equation}
Now adding $b_{k}$ times row $k$ to row $k+1$ for $k=1,2,\ldots,n-1$, followed by adding  $b_{k}$ times row $k$ to row $k+2$ for $k=1,2,\ldots,n-2$, and again continuing in this way removes all further dependence on $b_k$ for any $k$, so that
\begin{equation}\label{eqn-mtm}
  \det_{1\leq k,\ell\leq n}  \big(\, m_{k,\ell}(q_\ell)\,\big) 
	= \prod_{1\leq i<j\leq n} (b_i+a_j) \det_{1\leq k,\ell\leq n}  \big(\, \tilde{m}_{k,\ell}(q_\ell)\,\big) \,,
\end{equation}
where
\begin{equation}\label{eqn-tmLq}
	\tilde{m}_{k,\ell}(q_\ell) = L(\c,q_\ell)\ \ov{q}_\ell^{k}\ -  L(\c,-\ov{q}_\ell)\ (-q)_\ell^{k}\,.
\end{equation}
This provides the $\a$ and $\b$ dependence required in (\ref{eqn-detm-A}). 

Next we want to separate out the dependence on $\c$ and $\q$. First it should be noted
that $L(\c,q)$ is nothing other than the generating function for the elementary
symmetric functions $e_m(\c)$ in the sense that
\begin{equation}
    L(\c,q) = \prod_{i=1}^{n+1} (1+q\,c_i) = \sum_{m=0}^{n+1} q^m\, e_m(\c)\,.
\end{equation}
Hence
\begin{equation}\label{Eq-meq}
    \tilde{m}_{k,\ell}(q_\ell) = \sum_{m=0}^{n+1} e_m(\c)\ \left( q_\ell^{m-k}-(-1)^{m-k}q_\ell^{k-m}\right)\,.
\end{equation}
The term $m=k$ is identically zero and reorganising the remaining terms gives
\begin{equation}\label{Eq-tmeh}
    \tilde{m}_{k,\ell}(q_\ell) = (q_\ell+\ov{q}_\ell)\left( \sum_{r=0}^{k-1} (\!-\!1)^r e_{k-1-r}(\c) h_r(q_\ell,\!-\!\ov{q_\ell})
                        + \sum_{r=0}^{n-k} e_{k+1+r}(\c) h_r(q_\ell,\!-\!\ov{q}_\ell)\right).
\end{equation}
To obtain this, use has been made of the fact that
\begin{equation}
 \frac{q^{r+1}+(-1)^r q^{-r-1}}{q+q^{-1}}= \sum_{s=0}^r (-1)^s q^{r-2s}=\sum_{s=0}^r q^{r-s}(-\ov{q})^s
  = h_r(q,-\ov{q})
\end{equation}
for all $r\geq0$ and all non zero $q$, where $h_r(q,-\ov{q})$ is the complete
homogeneous symmetric function of the two parameters $q$ and $-1/q$. Hence
\begin{equation}\label{eqn-tmdeteh}
\begin{array}{l}
\ds \det_{1\leq k,\ell\leq n} \left(\tilde{m}_{k\ell}(q_\ell)\right) =\prod_{i=1}^n\ (q_i+\ov{q}_i)\cr
\ds \ \times\ \det_{1\leq k,\ell\leq n}\left( \sum_{r=0}^{k-1} (-1)^r e_{k-1-r}(\c) h_r(q_\ell,-\ov{q_\ell})
                        + \sum_{r=0}^{n-k} e_{k+1+r}(\c) h_r(q_\ell,-\ov{q}_\ell) \right)
\end{array}												
\end{equation}

To proceed we need a further Lemma:
\begin{appxlem}\label{lem-hr} 
For any $n\in\N$ let $\y=(y_1,y_2,\ldots,y_n)$ be a sequence of indeterminates, and let
$p,q$ be two non-zero indeterminates. Then for all $\y$ and all $r\geq 0$ 
\begin{equation}\label{eqn-hr}
h_r(p,-\ov{p},\y)-h_r(q,-\ov{q},\y) = (p-q)(1+\ov{p}\ov{q})\ h_{r-1}(p,-\ov{p},q,-\ov{q},\y)\,, 
\end{equation}
where, as usual, $h_{-1}(p,-\ov{p},q,-\ov{q},\y)=0$ for all $\y$ in the $r=0$ case.
\end{appxlem}

\begin{proof}
For any alphabet $\x=(x_1,x_2,\ldots,x_m)$ the generating function for $h_r(\x)$ takes the form
\begin{equation}
   \prod_{i=1}^m (1-t\, x_i)^{-1} = \sum_{r=0}^\infty\ t^r\, h_r(\x)\,.
\end{equation}
It follows that
\begin{equation}
\begin{array}{l}
\ds   \sum_{r=0}^\infty\ t^r\, \left( h_r(p,-\ov{p},\y) -  h_r(q,-\ov{q},\y) \right) \cr
\ds   = \left( ((1-t\,p)(1+t/p))^{-1} - ((1-t\,q)(1+t/q))^{-1} \right) \prod_{i=1}^n (1-t\, y_i)^{-1} \cr
\ds   = t\,(p-q)(1+1/pq)\ ((1-t\,p)(1+t/p)(1-t\,q)(1+t/q))^{-1} \prod_{i=1}^n (1-t\, y_i)^{-1} \cr
\ds   = t\,(p-q) (1+1/pq) \ \sum_{s=0}^\infty\ t^s\, h_s(p,-\ov{p},q,-\ov{q},\y)\,.
\end{array}
\end{equation}  
Comparing the coefficients of $t^r$ on both sides gives the required result (\ref{eqn-hr}) for all $r\geq 0$.
\end{proof}

Returning to our problem of evaluating the determinant on the right of (\ref{eqn-tmdeteh}),
subtracting column $\ell-1$ from column $\ell$ for $\ell=2,\ldots,n$ and using the above Lemma 
one can extract the common factors $(q_{\ell-1}-q_\ell)(1+\ov{q}_{\ell-1}\ov{q}_\ell)$ that are independent of 
the row number $k$ to leave terms in which each $h_r(q_\ell,-\ov{q}_\ell)$ is replaced by 
$-h_{r-1}(q_{\ell-1},-\ov{q}_{\ell-1},q_{\ell},-\ov{q}_\ell)$ for $\ell=2,\ldots,n$. It might be noted in particular
that that the number of terms is reduced since those in $h_0(q_\ell,-\ov{q}_\ell)$ go to $0$. 
Repeating this process of subtracting column $\ell-1$ from column $\ell$ but now for $\ell=3,\ldots,n$ gives 
rise to common factors $(q_{\ell-2}-q_\ell)(1+\ov{q}_{\ell-2}\ov{q}_\ell)$ and terms in 
$h_{r-2}(q_{\ell-2},-\ov{q}_{\ell-2},q_{\ell-1},-\ov{q}_{\ell-1},q_{\ell},-\ov{q}_\ell)$. Continuing one finds
\begin{equation}
\begin{array}{l}
\ds \det_{1\leq k,\ell\leq n} \left(\tilde{m}_{k\ell}(q_\ell)\right)
= \prod_{i=1}^n (q_i+\ov{q}_i) \prod_{1\leq i<j\leq n} (q_i-q_j)(1+\ov{q}_i\ov{q}_j) \cr
\ds \ \times\ \det_{1\leq k,\ell\leq n}\left( \sum_{r=0}^{k-\ell} (-1)^{r} e_{k-\ell-r}(\c) h_r^{(\ell)}(\q)
                        + \sum_{r=0}^{n+1-k-\ell} (-1)^{\ell-1} e_{k+\ell+r}(\c) h_r^{(\ell)}(\q) \right) \,.
\end{array}                        
\end{equation}
where we have set $h_r^{(\ell)}(\q):=h_r(q_1,\ov{q}_1,q_2,\ov{q}_2,\ldots,q_\ell,\ov{q}_\ell)$
for all $\ell=1,2,\ldots,n$. 

Remarkably the displayed determinant is independent of $\q$. To see this one should notice
that $n$th column contains only two non-zero terms namely $(-1)^{n-1}e_{n+1}(\u)$ in row $k=1$,
and $e_0(\u)$ in row $k=n$, where we have used the fact that $h_0^{(n)}(\q)=1$. 
Subtracting $(-1)^{n-\ell}h_{n-\ell}^{(\ell)}(\q)$ times this $n$th column 
from the $\ell$th column for all $\ell=1,2,\ldots,n-1$ eliminates all terms in $h_{n-\ell}^{(\ell)}(\q)$
from the $\ell$th column. This leaves the $(n-1)$th column bereft of any dependence on $\q$ since
$h_0^{(n-1)}(\q)=1$. One then
repeats the process subtracting $(-1)^{n-1-\ell}h_{n-1-\ell}^{(\ell)}(\q)$ times this $(n-1)$th column 
from the $\ell$th column for all $\ell=1,2,\ldots,n-2$ eliminating all terms in $h_{n-1-\ell}^{(\ell)}(\q)$.
Continuing, the only surviving terms are those in $h_r^{(\ell)}(\q)$ with $r=0$.
Thus
\begin{equation}
\begin{array}{l}
\ds \det_{1\leq k,\ell\leq n} \left(\tilde{m}_{k\ell}(q_\ell)\right)
= \prod_{i=1}^n q_i \prod_{1\leq i<j\leq n}(q_i-q_j) \prod_{1\leq i\leq j\leq n} (1+\ov{q}_i\ov{q}_j) \cr\cr
\ds \ \times\ \det_{1\leq k,\ell\leq n}\left( e_{k-\ell}(\c)\chi(k-\ell\geq0) + (-1)^{\ell-1} e_{k+\ell}(\c) \chi(k+\ell\leq n+1)\right)\,.
\end{array}
\end{equation}
where $\chi(P)$ is the truth function whereby $\chi(P)$ is $1$ if the proposition $P$ is true and $0$ otherwise.
This has separated out all the dependence on $\q$, leaving the dependence on $\c$ expressed as 
a determinant of elementary symmetric functions. Moreover the truth functions may be dropped since
$e_r(\c)=0$ for all integer $r<0$ and $e_r(\c)=0$ for all $r>n+1$ as $\c=(c_1,c_2,\ldots,c_{n+1})$
has only $n+1$ components. Hence 
\begin{equation}
\begin{array}{l}
\ds \det_{1\leq k,\ell\leq n} \left(\tilde{m}_{k\ell}(q_\ell)\right) 
= \prod_{i=1}^n q_i^{-n} \prod_{1\leq i<j\leq n}(q_i-q_j) \prod_{1\leq i\leq j\leq n} (1+ q_iq_j) \cr\cr
\ds ~~~~~~~~~~~~~~~~
\ \times \ \det_{1\leq k,\ell\leq n}\left( e_{k-\ell}(\c) + (-1)^{\ell-1} e_{k+\ell}(\c) \right)
\end{array}
\end{equation}
where the $\q$ dependent factors have been rearranged so as to show that they conform 
precisely with what is required in (\ref{eqn-detm-A}).
 
To complete our proof we need just one more lemma.
\begin{appxlem}\label{lem-A}
For $n\in\N$ and any sequence of indeterminates $\c=(c_1,c_2,\ldots,c_{n+1})$ we have
\begin{equation}\label{eqn-edet}
\det_{1\leq k,\ell\leq n}\left( e_{k-\ell}(\c) + (-1)^{\ell-1} e_{k+\ell}(\c)\right)
=\prod_{1\leq i<j\leq n+1}(1+c_ic_j) \,.
\end{equation}
\end{appxlem}

\begin{proof}
We follow a procedure introduced in~\cite{HamKin}, applied this time to a
determinant whose entries involve a signed sum of a pair of elementary symmetric functions 
rather than as previously a weighted sum of a pair of complete homogeneous symmetric functions.

The expansion of the determinant on the left of (\ref{eqn-edet}) yields
\begin{equation}\label{eqn-edet-kappa}
\sum_{r=0}^n \ \sum_{\kappa} (-1)^{\ell_1+\ell_2+\cdots+\ell_r-r} \det_{1\leq k,\ell\leq n}\left( e_{\kappa_\ell-\ell+k}(\c) \right)
\end{equation}
where the sum is over those $\kappa$ such that $\kappa_\ell=2\ell$ for $\ell\in\{\ell_1,\ell_2,\ldots,\ell_r\}$
with $n\geq \ell_1>\ell_2>\cdots>\ell_r\geq1$ and $\kappa_\ell=0$ otherwise. 
Clearly $\kappa$ is not a partition, but we may permute the columns of 
the determinant on the right, keeping track of the number of transpositions of
columns. This gives
\begin{equation}\label{eqn-edet-kappa2}
(-1)^{\ell_1+\ell_2+\cdots+\ell_r-r} \det_{1\leq k,\ell\leq n}\left( e_{\kappa_\ell-\ell+k}(\c) \right)
= \det_{1\leq k,\ell\leq n}\left( e_{\gamma_\ell-\ell+k}(\c) \right) = s_{\gamma'}(\c)\,
\end{equation}
where the second equality is just the dual Jacobi-Trudi identity,
and $\gamma$ is the partition given in Frobenius notation by
\begin{equation}
  \gamma=\left( \begin{array}{cccc} \ell_1&\ell_2&\cdots&\ell_r\cr  \ell_1-1&\ell_2-1&\cdots&\ell_r-1\cr \end{array} \right)\,.
\end{equation}
Thus we have $\gamma\in{\cal C}={\cal P}_{1}$, the set of all the partitions which in Frobenius notation 
are such that each arm length exceeds the corresponding leg length by $1$. In fact the sum over $\kappa$
yields all such partitions $\gamma$ of arm length at most $n$, since the only restriction is that
$n\geq \ell_1>\ell_2>\cdots>\ell_r\geq1$. It follows that
\begin{equation}\label{eqn-edet-gamma}
\begin{array}{l}
\ds \det_{1\leq k,\ell\leq n}\left( e_{k-\ell}(\c) + (-1)^{\ell-1} e_{k+\ell}(\c)\right) \cr\cr
\ds~~~~~~~~~~~~~ = \sum_{\gamma\in{\cal C};\gamma_1\leq n+1}\!\!\! s_{\gamma'}(\c)\
= \sum_{\alpha\in{\cal A};\ell(\alpha)\leq n+1}\!\!\! s_{\alpha}(\c)\
= \sum_{\alpha\in{\cal A}} s_{\alpha}(\c)
\end{array}
\end{equation}
where ${\cal A}={\cal P}_{-1}$ is the set of partitions conjugate to those in ${\cal C}={\cal P}_{1}$
and use has been made of the fact that  $s_{\alpha}(\c)=0$
for all $\alpha\in{\cal A}$ of length $\ell(\alpha)>n+1$ since $\c=(c_1,c_2,\ldots,c_{n+1})$ has
no more than $n+1$ components. 

Finally, from (\ref{eqn-ACalphagamma}) it follows that
\begin{equation}\label{eqn-edet-alpha}
\det_{1\leq k,\ell\leq n}\left( e_{k-\ell}(\c) + (-1)^{\ell-1} e_{k+\ell}(\c)\right) =  \prod_{1\leq i<j\leq n+1}  (1+c_ic_j) \,,
\end{equation}
as required, to complete the proof of Lemma~\ref{lem-A}.
\end{proof}

This in turn provides the required dependence on $\c$ of the right hand side of 
(\ref{eqn-detm-A}), thereby completing the proof of Lemma~\ref{lem-detm-A}.
\end{proof}

\end{document}